\documentclass[11pt,a4paper,reqno]{amsart}

\usepackage{amsmath,amssymb,amsthm,amsfonts}
\usepackage{amsbsy}
\usepackage[utf8]{inputenc}
\usepackage{geometry}
\usepackage{enumitem}
\usepackage[english]{babel} 
\usepackage{tikz}
\usepackage{color}
\usepackage{cancel}
\definecolor{LinkColor}{rgb}{0,0,0} 
\usepackage[colorlinks=true,linkcolor=LinkColor,citecolor=LinkColor,urlcolor= LinkColor, naturalnames, hyperindex, pdfstartview=FitH, bookmarksnumbered, plainpages]{hyperref} 
\usepackage{units}
\usepackage{amstext}
\usepackage{setspace}
\usepackage{mathtools}
\usepackage[normalem]{ulem}
\usepackage{placeins}
\usepackage{cancel}
\usepackage{units}
\usepackage{amstext}
\usepackage{stackrel}

\newtheorem{theorem}{Theorem}[section]
\newtheorem{corollary}[theorem]{Corollary}
\newtheorem{lemma}[theorem]{Lemma}
\newtheorem{proposition}[theorem]{Proposition}

\newtheorem{nr}[theorem]{}

\theoremstyle{definition}

\newtheorem{remark}[theorem]{Remark}

\newtheorem{problem}[theorem]{Problem}

\newcommand{\SL}{\operatorname{SL}}
\newcommand{\GL}{\operatorname{GL}}
\newcommand{\Tr}{\operatorname{Tr}}

\newcommand{\V}{\textup{V}}
\newcommand{\ZZ}{\mathbb{Z}}
\newcommand{\QQ}{\mathbb{Q}}
\newcommand{\C}{\textup{C}}
\newcommand{\Irr}{\operatorname{Irr}}
\newcommand{\sgn}{\operatorname{sgn}}
\renewcommand{\O}{\mathcal{O}}
\newcommand{\lra}{\longrightarrow}
\numberwithin{equation}{section}

\newcommand{\boxedcomp}[1]{\begin{tabular}{|c|} \hline \ensuremath{#1} \\ \hline \end{tabular}}
\newcommand{\boxxedcomp}[2]{\begin{tabular}{|c|} \hline \ensuremath{#1} \\ \hline \ensuremath{#2} \\ \hline \end{tabular}}

\title{On the first Zassenhaus conjecture and direct products }
\date{\today}

\author{Andreas B\"achle}
\address{(Andreas B\"achle) Vakgroep Wiskunde, Vrije Universiteit Brussel, Pleinlaan 2, 1050 Brussels, Belgium}
\email{\href{mailto:abachle@vub.ac.be}{abachle@vub.ac.be}}

\author{Wolfgang Kimmerle}
\address{(Wolfgang Kimmerle) Fachbereich Mathematik, IGT,  Universit\"{a}t Stuttgart, Pfaffenwaldring 57, 70550 Stuttgart, Germany}
\email{\href{mailto:kimmerle@mathematik.uni-stuttgart.de}{kimmerle@mathematik.uni-stuttgart.de}}

\author{Mariano Serrano}
\address{(Mariano Serrano) Departamento de Matem\'aticas, Universidad de Murcia, 30100, Murcia, Spain}
\email{\href{mailto:mariano.serrano@um.es}{mariano.serrano@um.es}}

\thanks{The first author is a postdoctoral researcher of the FWO (Research Foundation Flanders). 
The third author has been partially supported by the Spanish Government under Grant MTM2016-77445-P with "Fondos FEDER" and, by Fundaci\'on S\'eneca of Murcia under Grant 19880/GERM/15.}

\subjclass[2010] {16S34, 16U60, 20C05} 
\keywords{integral group ring, torsion units, Zassenhaus Conjecture, direct products, Frobenius groups, HeLP method}

\begin{document}

\maketitle

\begin{abstract}

In this paper we study the behavior of the first Zassenhaus conjecture (ZC1)
under direct products as well as the General Bovdi Problem (Gen-BP)
which turns out to be a slightly weaker variant of (ZC1).
Among others we prove that (Gen-BP) holds for Sylow tower
groups, so in particular for the class of supersolvable groups.

(ZC1) is established for a direct product of Sylow-by-abelian groups
provided the normal Sylow subgroups form together a Hall subgroup. We also show (ZC1)
for certain direct products with one of the factors a Frobenius group.

We extend the classical HeLP method to group rings with coefficients from any ring of algebraic integers.
This is used to study (ZC1) for the direct product $G \times A$, where $A$ is a finite abelian group and $G$ has order at most 95.
For most of these groups we show that (ZC1) is valid and for all of them
that (Gen-BP) holds. 
Moreover, we also prove that (Gen-BP) holds for the direct product of a Frobenius group with
any finite abelian group.

\end{abstract}

\section{Introduction}

Let $G$ be a finite group. Denote its integral group ring by $\ZZ G$
and let $\V(\ZZ G)$ be the group of normalized units of $\ZZ G$  (i.e. units with augmentation $1$). A long-standing conjecture of
H.~Zassenhaus \cite{Zassenhaus} (see also \cite[Section~37]{SehgalBook}) is as follows:

\vskip1em

\textbf{First Zassenhaus conjecture (ZC1).} Every torsion unit of $\ZZ G$ is conjugate to an element of $\pm
G$ in the units of $\mathbb{Q}G$. 

\vskip1em

As usual we say that (ZC1) or the first Zassenhaus conjecture holds for a particular group or a class of
finite groups if it is correct for that group or for each group of
this class.\footnote{As the name suggests, there are also a second and a third Zassenhaus conjecture dealing with conjugacy of (not necessarily cyclic) torsion subgroups of $\ZZ G$. (See \cite[Section~37]{SehgalBook} for more details.)}

(ZC1) holds for finite nilpotent groups \cite{Weiss91}, for
Sylow-by-abelian groups (i.e.\ groups having
a normal Sylow $p$-subgroup with abelian complement) \cite{Hertw06} and
for cyclic-by-abelian groups \cite{CMdR}. It has been
established as well for several non-solvable groups,
cf.\ \cite{BKM}. With algorithmic tools, mainly with the HeLP
method, (ZC1) has been shown for all groups of order at most $143$
\cite{BHKMS}. Very recently F.~Eisele and L.~Margolis announced
a metabelian counterexample to (ZC1) \cite{EisMar}. 

One main
object of this article is to study the behavior of (ZC1) under direct
products. 
Let $G$ and $H$ be finite groups and assume that (ZC1) holds for $G$ and $H$.
Then very little is known whether (ZC1) holds for $G
\times H$. If $H$ is an elementary abelian $2$-group this was answered affirmatively by C.~H\"ofert, cf.\ \cite{Hoe} or \cite{HoeKim}.   
M.~Hertweck proved that (ZC1) holds for $G \times H$ provided $G$ is
nilpotent and $H$ is an arbitrary finite group for which (ZC1) is known and whose order is coprime to
$|G|$ \cite[Proposition 8.1]{Hertweck08}.  

\vskip1em

In Section \ref{sec:nilpotent-by-abelian} we show that if $G$ is a direct product of Sylow-by-abelian groups then
(ZC1) holds for $G$ provided the normal Sylow subgroups form a Hall subgroup. 
Section \ref{sec:Frobenius} deals with Frobenius and Camina groups. Among other we show 
that (ZC1) holds for a direct product of a Frobenius group  with
metacyclic complements and a finite abelian group or of such a Frobenius
group with an arbitrary finite group of coprime order.

\vskip1em

In view of the example constructed by Eisele and Margolis, weaker
versions (you may also say substitutes) of the first Zassenhaus conjecture will come into the middle of
research. For a recent overview we recommend
\cite{MdR}. Two of them arise naturally in the context of our
article (the abbreviations are taken from \cite{MdR}). 

\vskip1em
The first one has been posed by the second author in \cite[Question
22]{Ari}.
\vskip1em
\textbf{(KP).}  Given a torsion element $u$ in $\V(\ZZ G)$, 
is there a finite group $H$ containing $G$ as subgroup such that
$u$ is conjugate in the units of $\mathbb{Q} H$ to an element of $G$? 

\vskip1em
The second one has been stated in the case when $n$ is a prime power
order by A.A.~Bovdi \cite[p.~26]{Bov}.
\vskip1em

\textbf{(Gen-BP).} 
Let $u$ be an element of $\V(\ZZ G)$
of order $n$. For a positive integer $m$, denote by $\varepsilon_{G[m]}(u)$ 
the sum of the coefficients of $u$ at all elements of order $m$ of
$G$. Is $\varepsilon_{G[m]}(u)=0$ for all $m \neq n$? 

\vskip1em

In \cite{MdR} it is shown that (KP) and (Gen-BP) are in the form
stated as above equivalent. Certainly concerning (KP) it is of interest to construct $H$
as small as possible and whether (ZC1) holds for $H$. 
In Section 3 we give a positive answer to (KP) in the case
when $G$ has a normal nilpotent Hall subgroup with abelian
complement. Here $G$ is embedded into a suitable direct product of
Sylow-by-abelian groups for which (ZC1) is valid (see Corollary~\ref{ZassenhausLargerGroup}). This covers in particular the groups
constructed by Eisele and Margolis and shows that for integral group
rings the property (ZC1)
is not closed under group rings of subgroups. In contrast to this, (KP) and
therefore (Gen-BP) as well are inherited by group rings of subgroups. Moreover we show 
in Section 3 that (Gen-BP) holds when $G$ is a Sylow tower group. In particular it holds for
supersolvable groups. 

\vskip1em

The case of groups of the form $G \times A$ where (ZC1) holds for $G$ and $A$ is finite abelian leads naturally to the
problem to consider the analogue of the first Zassenhaus conjecture for group
rings $\O G$ where $\O$ is a ring of algebraic integers (see Proposition~\ref{prop:Hertweck_reduction}). Extending the HeLP method to this situation we investigate 
this question when $G$ has order at most 95. Surprisingly this leads 
for many groups of small order to a positive result (see Proposition~\ref{ExtendedResult}). However in $\ZZ [i]S_4$ the case of a normalized unit of order $4$
remains. In a certain $\ZZ$-order containing $\ZZ [i]S_4$ there are such units arising from units of
$\ZZ [i, 1/2]S_4$ and we will give one such unit explicitly.
A more careful analysis of this case requires large calculations and will be treated in an extra article.

\section{Preliminaries and the HeLP method}\label{sectionHeLPMethod}

In this section we recall known results about torsion units
in group rings which we use in the following sections and explain briefly the HeLP method (for a more detailed discussion see \cite{HeLPOverview, HertweckPa, AngelMariano}).

Let $\ZZ_{\geq 0}$ denote the set of non-negative integers and $\ZZ_p$ the $p$-adic integers for a prime $p$. For a positive integer $n$, we always use $\zeta_n$ to denote a complex primitive $n$-th root of unit. If $K/F$ is a Galois extension, then $\Tr_{K/F} : K \rightarrow F$ denotes the trace map.   

Let $G$ be a finite group and $g\in G$. 
We denote by $g^G$ the conjugacy class of $g$ in $G$ and by $G'$ the commutator subgroup of $G$. 
If $p$ is a prime and $G$ has a normal Sylow $p$-subgroup, then this subgroup is denoted by $G_p$; if it has a normal Hall $p'$-subgroup,
then this group is denoted by $G_{p'}$.
For $\pi$ a set of primes, the $\pi$-part and the $\pi'$-part of $g$ are denoted by $g_\pi$ and $g_{\pi'}$, respectively. 
By $o(x)$ we denote the order of the group element $x$.

If $u=\sum_{h\in G}u_h h$ is an element of a group ring of $G$ then the \emph{partial augmentation} of $u$ at $g\in G$ is 
$$\varepsilon_{g}(u)=\sum_{h\in g^G}u_h.$$ For a positive integer $m$ we will denote the \emph{generalized trace} of $u$ (as it appears in (Gen-BP)) by 
$$\varepsilon_{G[m]}(u)=\sum_{\substack{h\in G \\ o(h)=m}}u_h.$$
Clearly, $\varepsilon_{G[m]}(u)$ is just the sum over all partial augmentations of $u$ at conjugacy classes containing elements of order $m$.

We will use the following elementary observation. 
Let $N$ be a normal subgroup of $G$ and set $\bar{G}=G/N$. 
For a torsion unit $u$ in $\ZZ G$, we shall extend the bar convention when writing $\bar{u}$ for the image of $u$ under the natural map $\ZZ G \rightarrow \ZZ \bar{G}$. 
To express that $x$ is conjugate to the element $y$ in the group $G$ we use the notation $x\sim_G y$; if the group is clear from the context the subscript $G$ will be dropped. Since any conjugacy class of $G$ maps onto a conjugacy class of $\bar{G}$, we have for any $x\in G$: 
\begin{equation}\label{ProjectionPA}
\varepsilon_{\bar{x}}(\bar{u})=\sum_{g^G,\ \bar{g}\sim \bar{x}} \varepsilon_{g}(u);
\end{equation}
where $\sum_{g^G}$ is an abbreviation of $\sum_{t\in T}$ for $T$ a set of representatives of the conjugacy classes of $G$.

The following collects some well known facts about torsion units in
integral group rings (see \cite[Corollary 1.5, Theorem 7.3, Lemma
41.5]{SehgalBook},\cite{CohLiv}, \cite{MRSW} and 
\cite[Theorem~2.3]{HertweckPa}). 

\begin{nr}\label{knownfacts}
        Let $G$ be a finite group and let $u$ be an element of $\V(\ZZ G)$ of order $n$. Then the following statements hold:
        \begin{enumerate}
                \item If $u\ne 1$ then $\varepsilon_1(u)=0$ (Berman-Higman Theorem). 
                \item $n$ divides the exponent of $G$.  
                \item \label{HertweckOrder} If $g\in G$ and $\varepsilon_g(u)\ne 0$ then the order of $g$ divides the order of $u$.
                \item \label{MRSW} $u$ is rationally conjugate to an element of $G$ if and only if 
                $\varepsilon_{g}(u^d)\geq 0$ for every $g\in G$ and
                every $d\mid n$. 
                \item \label{u_p_element} If $N$ is a normal $p$-subgroup of $G$ and $u$ maps
                  under the map $\ZZ G \lra \ZZ G/N$ to 1, then $u$ is
                  a $p$-element.   
        \end{enumerate}
\end{nr}

We will use these results without further mention.
 We say that a normalized unit $u$ has trivial partial augmentations if the partial augmentations of $u^d$, for all divisors $d$ of the order of $u$ coincide with partial augmentations of a fixed group element $g \in G$. In this case $u$ is rationally conjugate to the element $g$ by \ref{knownfacts}.\eqref{MRSW}.

Let $\rho$ be a representation of $G$ affording the character $\chi$ and  
let $u$ be a unit of $\ZZ G$ of order $n$. 
As $u^n =1$, we have that every eigenvalue of $\rho(u)$ is of the form $\zeta_n^\ell$ for some integer $0\leq \ell \leq n-1$ and the following formula gives the multiplicity of $\zeta_n^\ell$ in terms of the partial augmentations of $u$ (see \cite{LP89}): 
\begin{equation}\label{equationHELP}
\mu_\ell(u,\chi) = \frac{1}{n}\sum_{x^G}\sum_{d\mid n}\varepsilon_x(u^d)\Tr_{\QQ(\zeta_n^d)/\QQ}(\chi(x)\zeta_n^{-\ell d})\in \ZZ_{\geq 0}. 
\end{equation} 
Observe that (\ref{equationHELP}) makes sense because of \ref{knownfacts}.(\ref{HertweckOrder}). This formula can also be extended to Brauer characters modulo a prime not dividing $n$ (see \cite{HertweckPa}) and it is the bulk of the HeLP method. 

Let $\ZZ_{(G)}$ be the intersection of the localizations $\ZZ_{(p)}$, with $p$ a prime dividing the order of $G$. 
In several papers Hertweck considered the behavior of torsion
units of integral group rings mapping to the identity under the map
$\ZZ G \rightarrow \ZZ G/N ,$ where $N$ is a normal $p$-subgroup of $G$. His
main results are the following (see \cite[Theorem~B]{Hertweck13}, \cite{MarHer} for a proof and \cite[Lemma~2.2]{Hertweck08}, respectively): 

\begin{theorem}\label{HertweckZ_p}
        Let $N$ be a normal $p$-subgroup of a finite group $G$. Then any torsion unit in $\ZZ G$ which maps to the identity under the natural map $\ZZ G\rightarrow \ZZ G/N$ is conjugate to an element of $N$ by a unit in $\ZZ_p G$.
\end{theorem}

\begin{proposition}\label{HertweckPaugmentation}
        Let $G$ be a finite group and $p$ a prime integer. Let $R$ be a $p$-adic ring and $u$ a torsion unit of $RG$ with augmentation one.  Suppose that the $p$-part of $u$ is conjugate to an element $x$ of $G$ in the units of $RG$ and $g$ is an element of $G$ such that the $p$-parts of $x$ and $g$ are not conjugate in $G$. Then $\varepsilon_g(u)=0$.
\end{proposition}

It is not known whether the first Zassenhaus conjecture behaves well with respect to quotient groups (for subgroups and extensions see Remark \ref{not_subgroup_closed} below). As mentioned in the introduction, also for direct products barely anything is known. The only easy observation one can make is the following. 

\begin{remark}\label{rem:ZC_for_direct_factors} The first Zassenhaus conjecture is inherited by direct factors. That is, if the first Zassenhaus conjecture is known for $G = H \times K$, the direct product of $H$ and $K$, then the first Zassenhaus conjecture has to hold for $H$ and $K$.

Denote by $\iota \colon \QQ H \to \QQ G$ and $\pi \colon \QQ G \to \QQ H$ the ring homomorphisms induced by the inclusion of $H$ into $G$ and the projection of $G$ onto $H$, respectively. Assume that $u \in \V(\ZZ H)$ is a torsion element. Then also $\iota(u) \in \V(\ZZ G)$ is a torsion element and by assumption conjugate to an element $g \in G$ by a unit $x \in \QQ G$. But then $u = \pi(\iota(u))$ is conjugate by the unit $\pi(x) \in \QQ H$ to $\pi(g) \in H$. 
\end{remark}

\section{Nilpotent-by-Abelian Groups}\label{sec:nilpotent-by-abelian}

The following lemma is a slight generalization of \cite[Lemma 5.5]{Hertw06} for our purpose.

\begin{lemma}\label{lemma:pi-conj-classes} 
        Let $G$ be a finite group and $\pi$ be a set of primes.  
        Suppose that $N$ is a normal nilpotent Hall $\pi$-subgroup of $G$ with abelian complement $K$. 
        Let $x \in N$ be a $\pi $-element and let $k \in K$. 
        Then
                $$C = \{g \in G \; :\; g_{\pi} \sim x \text{ and } g = n  k \mbox{ for some } n \in N \}$$ 
        forms a conjugacy class of $G$.  
\end{lemma}

\begin{proof} Let $f, g \in C$. 
        As $K$ is abelian, we may after conjugation assume that $f_{\pi} = g_{\pi} = x$, $f = x \cdot n_1  k$ and  $g = x  \cdot n_2  n_1  k$ with $n_1, n_2 \in N$. 
        Let $H = \langle n_1k, n_2 \rangle$. 
        Having in mind that $n_1k$ and $n_2n_1k$ are the $\pi'$-parts of $f$ and $g$ respectively, we get that $x\in \C_G(H)$. 
        Let $M = N \cap H$.  Then $M$ is a normal subgroup of $H$ and $K_1 = \langle n_1k \rangle$ and $K_2 = \langle n_2n_1k \rangle $ are  complements to $M$ in $H$. 
        By the Schur-Zassenhaus Theorem we get that $K_1^h = K_2$ for some $h \in H$. 
        As $(n_1k)^h \in Nk$ we get  
        $$ (n_1k)^h  = n_2n_1k.$$ 
    It follows that 
        $$ f^h = x^h (n_1k)^h = x n_2 n_1 k = g,$$ 
        which establishes the lemma. 
\end{proof}

\begin{theorem}\label{theo:main}
        Let $G=\left(P_1 \rtimes A_1 \right) \times \dots \times \left(P_k \rtimes A_k\right)$ be a finite group where $A_j$ is an finite abelian group for every $j$ and $P_1 \times \dots \times P_k$ is a Hall subgroup of $G$. 
        Then the first Zassenhaus conjecture holds for $G$. 
\end{theorem}

\begin{proof}
        Let $u$ be a torsion element of $\V(\ZZ G)$. We will show that all  partial augmentations of $u$ but one vanish. 
        Let bars denote reduction modulo $N=P_1\times \dots \times P_k$ and note that all torsion units of $\ZZ \bar{G}$ are trivial as $\bar{G}$ is abelian.   
        Denote by $\pi = \pi(N)$ the set of prime divisors of the order of $N$. 
        
        For any $p\in \pi$, let $u_p$ be the $p$-part of $u$ and  let $P$ be the Sylow $p$-subgroup of $G$.  
        Then $u_p$ maps to $1$ under the natural map $\ZZ G\rightarrow \ZZ G/P$ as $G/P$ is a $p'$-group. 
        Thus $u_p$ is conjugate in the units of $\ZZ_pG$ to an element $x_p\in P$ by Theorem \ref{HertweckZ_p}. 
        Using Proposition \ref{HertweckPaugmentation} we deduce that $\varepsilon_g(u)=0$ for every $g\in G$ whose $p$-part is not conjugate to $x_p$. As $N$ is nilpotent, applying this argument for every $p\in \pi$  we get that $\varepsilon_g(u)=0$ for every $g\in G$ whose $\pi$-part is not conjugate to $x=\prod_{p\in \pi} x_p\in N$. 
        
        Take any $h \in G$. 
        The partial augmentation $\varepsilon_{\bar{h}}(\bar{u})$  is the sum of all partial augmentations $\varepsilon_g(u)$ with $g\in G$ and $\bar{g} = \bar{h}$ in $\bar{G}$ (since $\bar{G}$ is abelian). 
        By the previous paragraph, we need to sum only over conjugacy classes of elements $g\in G$ whose $\pi$-part is conjugate to $x$ (and $\bar{g} = \bar{h}$ of course). By Lemma~\ref{lemma:pi-conj-classes}, this sum extends over a single conjugacy class (if any).
        
        Thus, $\varepsilon_{\bar{g}}(\bar{u}) = \varepsilon_{g}(u)$ for all $g \in G$ whose $\pi$-part is conjugate to $x$. By the
        Berman-Higman Theorem, $\varepsilon_k (\bar{u}) \ne 0$ for exactly one $k \in \bar{G}$. 
        Applying Lemma \ref{lemma:pi-conj-classes} again, we see that there is only one partial augmentation of $u$ different from $0$, as desired.
\end{proof}

\begin{corollary}\label{ZassenhausLargerGroup} Assume that the finite group $H$ has a normal nilpotent Hall subgroup $N$ such that $H/N$ is abelian. Then $H$ can be embedded into a group $G$ for which the first Zassenhaus conjecture holds. In particular (KP) has an affirmative answer for $H$. \end{corollary}

\begin{proof} Let $A \simeq H/N$ be a complement of $N$ in $H$ and $N = \prod_{j=1}^k P_j$ the decomposition of $N$ as direct product of its Sylow subgroups. Then set $G = N \rtimes \left(\prod_{j=1}^k A\right)$, where the $j$th factor of $\prod_{j=1}^k A$ acts on $P_j$ like the complement $A$ in $H$ and trivial on all other Sylow subgroups of $N$. Then \[ H = NA \hookrightarrow G =N \rtimes \left(\prod_{j=1}^k A\right) \colon na \mapsto (n, a, ..., a) \] is an embedding of $H$ into $G$. Note that the first Zassenhaus conjecture holds for $G$ by Theorem~\ref{theo:main}.  

It remains to prove that (KP) has a positive answer for $H$. As $H$ can be embedded into the group $G$ for which (ZC1) holds, we have that each torsion element $u \in \V(\ZZ H)$ is conjugate within $\QQ G$ to an element $g \in G$. This means $\varepsilon_g(u) \not= 0$. But as then necessarily $g^G \cap H \not= \emptyset$, $u$ is also conjugate to an element of $H$ within $\QQ G$. \end{proof}

\begin{remark}\label{not_subgroup_closed} The groups constructed by
  Eisele and Margolis in \cite{EisMar} as counterexamples to (ZC1)
  have normal abelian Hall subgroup with abelian complement, so
  Corollary \ref{ZassenhausLargerGroup} shows that these groups can be
  embedded as normal subgroups in a group for which (ZC1) holds. This shows that the property (ZC1) is
  not closed under taking subgroups, not even under taking normal
  subgroups. In contrast to this, (KP) is clearly a subgroup closed
  property. As (ZC1) holds 
for abelian groups, (ZC1) can also not be an extension closed property. \end{remark}

\begin{proposition} \label{normalSylowp} Suppose that the finite group
  $G$ has a normal Sylow $p$-subgroup
        $P$ and that (Gen-BP) holds for $G/P$. Then (Gen-BP)
        holds for $G$. \\ 
Moreover
 if $G$ has a normal Hall
         subgroup $N$ which is a Sylow tower subgroup and (Gen-BP) holds for $G/N$ then it also holds for $G$.
 
\end{proposition}
     
\begin{proof} Let $u \in \V(\ZZ G)$ be a torsion element. If $u$ is a $p$-element
          then $u$ is conjugate in $\ZZ_p G$ to an element of
          $G$ by Theorem~\ref{HertweckZ_p}. Thus there is precisely
          one partial augmentation of $u$ which is not zero. 

Assume $o(u) = p^m \cdot a$ with $a > 1$ and $p$ does not divide $a$.
Denote by $u_p $ the $p$-part of $u$. Then $u_p$ maps to the
identity under the quotient map $\ZZ G \rightarrow \ZZ G/P$. Again by 
Theorem~\ref{HertweckZ_p}, $u_p$ is conjugate within $\ZZ_p G$ to $k
\in G$. Now by Proposition~\ref{HertweckPaugmentation} each partial
augmentation $\varepsilon_h(u) = 0$ if the $p$-part $h_p$ is not
conjugate to $k$. Thus $\varepsilon_g(u) \neq 0$ implies that $g$ has
order $p^m \cdot b$ where $b$ divides $a$. 

Let $\sigma $ be the
reduction map from $G$ onto $G/P$ and denote by $\bar{u}$ the image
of $u$ under the induced map $\ZZ G$ onto $\ZZ G/P$.
Clearly, for any $g\in G$ we have that $\sigma (g)$ has order $b$ if and only if $g$ has order $p^k \cdot b$ for some integer $k$. Thus
$$ \varepsilon_{G[p^m \cdot b]} (u) = \varepsilon_{\bar{G}[b]} (\bar{u}).$$
By assumption, the right hand side is zero if and only if $b \neq a$.
Thus (Gen-BP) follows for $G$.

The second statement follows immediately from the first one by
induction on the number of primes dividing $|N |.$           
\end{proof}

\begin{theorem} \label{prop:corollary_sylowtower}
(Gen-BP) holds for finite groups with a Sylow tower. In particular 
it is valid for supersolvable groups. 
\end{theorem}

\begin{proof} 
  Supersolvable groups are Sylow tower groups \cite[VI,
    Satz~9.1]{Huppert67}. Thus the result follows from Proposition
    \ref{normalSylowp}. 
 \end{proof}

\begin{remark} \begin{itemize}
\item[a)] It also follows from Proposition \ref{normalSylowp}
  that (Gen-BP) holds for $G$ provided $G$ has a normal nilpotent Hall subgroup with abelian
  complement. As (Gen-BP) is equivalent to (KP) (see \cite{MdR}), this
  gives for such a group a proof for (KP) which of course also follows for such groups from
  Corollary \ref{ZassenhausLargerGroup}. Note however that (ZC1) for
  the larger group need not hold if (KP) is valid. 
  \item[b)]
  Further examples of Sylow tower groups $G$ are finite groups having
 a nilpotent normal subgroup $N$ such that $G/N$ is a $p$-group. Vice
 versa, groups $G$ with a normal $p$-subgroup $P$ and nilpotent
 quotient $G/P$ are Sylow tower groups.
 From Burnside's Transfer Theorem \cite[IV, Satz~2.7]{Huppert67} it follows that finite solvable groups all of whose Sylow subgroups are abelian
with different invariants have a Sylow tower.
\end{itemize}
\end{remark}

\section{Frobenius Groups}\label{sec:Frobenius}

A finite group $G$ is called a \emph{Camina group} if $G\ne G'$ and $gG' = g^G$ for all $g \in G \setminus G'$. 
Camina groups found a lot of attention since they were introduced by A.R.~Camina in 1978. 
All Camina groups were described by R.~Dark and C.M.~Scoppola (the capstone can be found in \cite{DarkScoppola, Lewis14} and for the last gap that was closed see \cite{isaacslewis}). 
We collect this classification as follows. 

\begin{theorem}\label{classificationCamina}
A finite non-abelian group is a Camina group if and only if it is either a Camina $p$-group (which is necessarily of nilpotency class at most $3$) or a Frobenius group whose complement is cyclic or isomorphic to $Q_8$.      
\end{theorem}

In order to study the first Zassenhaus conjecture for the direct product of a Camina group and an finite abelian group, we first verify the first Zassenhaus conjecture for Camina groups.

\begin{proposition}\label{CaminaGroups}
        The first Zassenhaus conjecture holds for Camina groups.
\end{proposition}

\begin{proof}
        Let $G$ be a Camina group. 
        As the first Zassenhaus conjecture holds for abelian groups, we may assume that $G$ is non-abelian.
        We use the classification of Camina groups stated in Theorem \ref{classificationCamina}. 
        If $G$ is a Camina $p$-group, then in particular it is nilpotent and hence the result follows from \cite{Weiss91}.  
        Suppose that  $G$ is a Frobenius group whose complement is cyclic or isomorphic to $Q_8$ and 
        let $u \in \V(\ZZ G)$ of order $k$.  
        By \cite[Theorem 2.1]{JuriaansPolcino00} we know that $k$ divides either the order of the Frobenius kernel $N$ of $G$ or the order of the Frobenius complement of $G$. 
        In the first case, $u$ gets mapped to $1$ under the natural ring homomorphism induced by modding out $N$ and the result follows from  \cite{Hertweck12}. 
        Otherwise, $k$ is relatively prime to the order of $N$ and the result follows from \cite[Corollary 2.3]{DoJu}.
\end{proof}

\begin{proposition}\label{FrobeniusXAbeliano}
        Let $A$ be any finite abelian group and $F$ a Frobenius group with Frobenius complement $C$. If the first Zassenhaus conjecture holds for $C\times A$ then it also holds for $F \times A$. 
\end{proposition} 

\begin{proof}
        Using Remark~\ref{rem:ZC_for_direct_factors} and the fact that the first Zassenhaus conjecture holds for $C\times A$ by assumption, we deduce that it also holds for $C$. We claim that the first Zassenhaus conjecture holds for the Frobenius group $F = N \rtimes C$. 
        Indeed, by \cite[Theorem~2.1]{JuriaansPolcino00} the order of an element of $\V(\ZZ F)$ is either a divisor of the order of $N$ or a divisor of the order of $C$. 
        In the first case, the unit maps to the identity under the natural homomorphism $\ZZ F \to \ZZ F/N$ and hence it is rationally conjugate to an element of $F$ by \cite[Theorem]{Hertweck12}.  
        In the second case, it is conjugate to a unit of $\ZZ C$ by a unit of $\QQ F$ by \cite[Theorem~(37.17)]{SehgalBook} and hence eventually rationally conjugate to an element of $C$ by Remark~\ref{rem:ZC_for_direct_factors} (see also the discussion at the end of the introduction in \cite{Hertweck12}). 
        This finishes the proof of the claim. We will use it without further mention. 

        Let $G=\left(N\rtimes C\right) \times A$ and $u\in \V(\ZZ G)$ be a torsion element. 
        We will prove that all partial augmentations of $u$ are non-negative and after obtaining that, the result will follow by \ref{knownfacts}.(\ref{MRSW}). 
        We argue by induction on $o(u)$ and on $|G|$. Note that the
        Frobenius kernel $N$ is nilpotent by a famous result of J.~Thompson, so it has a normal Sylow subgroup for each prime divisor $p$.
        
        Clearly, for every prime $p\mid |N|$, the first Zassenhaus conjecture holds for $G/N_p$ by induction on $|G|$.
        Therefore, if there is a prime $p\mid |N|$ such that $p\nmid
        o(u)$ we can use equation \eqref{ProjectionPA} on page~\pageref{ProjectionPA} with normal
        subgroup $N_p$ and \ref{knownfacts}.(\ref{HertweckOrder}) to obtain for every $x\in G$ that $\varepsilon_{\bar{x}}(\bar{u})=\sum_{g^G,\bar{g}\sim \bar{x}}\varepsilon_{g}(u)=\varepsilon_{x}(u)\geq 0$, as desired.      
        
        Thus we may assume that every prime dividing $|N|$ also divides $o(u)$.  
        Moreover, for every prime $p\mid |N|$ we have that $p\nmid |C|$ because $N\rtimes C$ is Frobenius and using Theorem \ref{HertweckZ_p} with $N_p \times A_p$ we obtain $u_p\sim_{\ZZ_pG} n_pa_p\in N_p\times A_p$.
        
        Suppose that $n_p=1$ for some prime $p\mid |N|$. Then $u_p\sim_{\ZZ_p G}a_p$, and as $a_p$ is central in $G$, we get $u_p=a_p\in A_p$. 
        Let $v=u\cdot a_p^{-1}\in \V(\ZZ G)$ be a torsion element. 
        Thus $v_p=u_pa_p^{-1}=1$ and $v_q=u_q$ for every prime $q\ne p$. 
        This implies $o(v)\mid o(u)$ and $p\nmid o(v)$. 
        By induction on $o(u)$, all partial augmentations of $v$ are non-negative, and thus also all of $u$.    
        
        Therefore we may assume that $n_p\ne 1$ for every prime $p\mid |N|$.  
        Suppose that $\varepsilon_g(u)\ne 0$ for some $g\in G$. We claim that $g_q\in N\times A$ for every prime $q$.  
        Using Proposition \ref{HertweckPaugmentation} we deduce for every prime $p\mid |N|$ that $g_p\sim_G n_pa_p$. 
        This implies that $g_p\in N\times A$ for every prime $p\mid |N|$. 
        Let $q\mid |C\times A|$ be a prime.      
        If $q\nmid |C|$ then clearly $g_q\in N\times A$.        
        Suppose now that $q\mid |C|$. 
        Write $g_q=h_qa_q$ with $h_q\in N\rtimes C$ and $a_q\in A$. 
        As $g_p\sim_G n_pa_p$ and $a_p$ is central in $G$ for any prime $p\mid |N|$, there is $f\in N\rtimes C$ such that $g_p = f^{-1}n_pfa_p$. 
        Let $b_p = f^{-1}n_pf\in N\rtimes C$. 
        If $h_q\ne 1$ then $g_pg_q=b_pa_ph_qa_q=b_ph_qa_pa_q$ where $1\ne b_ph_q\in N\rtimes C$ and $pq\mid o(b_ph_q)$ because 
        $b_ph_q = b_pa_ph_qa_qa_p^{-1}a_q^{-1}=g_pg_qa_p^{-1}a_q^{-1}=g_qg_pa_p^{-1}a_q^{-1}=h_qb_p$, contradicting the fact that $N\rtimes C$ is Frobenius.
        Therefore,  if $q\mid |C|$ then $h_q=1$ and hence $g_q=a_q\in A_q\subseteq N\times A$. 
        This finishes the proof of the claim.
        
        As a consequence of the claim we get that $\varepsilon_g(u)\ne 0$ implies $g_q\sim_G n_qa_q\in N\times A$ for every prime $q$. 
        By induction on $|G|$, we can fix a prime $r$ dividing $|A|$. 
        Write $G=\left(\left(N\rtimes C\right)\times A_{r'}\right)\times A_r$ and let $x\in \left(N\rtimes C\right)\times A_{r'}$. 
        The first Zassenhaus conjecture holds for $G/A_r$ by induction on $|G|$. 
        So using \eqref{ProjectionPA} on page~\pageref{ProjectionPA} with $A_r$ we get 
        $$\varepsilon_x(\bar{u})=\sum_{b\in A_r}\varepsilon_{xb}(u)=\varepsilon_{xa_r}(u)\geq 0,$$
        as desired.
\end{proof}

\begin{corollary}
        The first Zassenhaus conjecture holds for the direct product $G\times A$ where $A$ is any finite abelian group and $G$ is either a Camina group or a Frobenius group whose complement has odd order. 
\end{corollary}
\begin{proof}
        By \cite[V.8.18]{Huppert67} (or by \cite[Theorem 18.1]{Passman68}) we know that odd order Frobenius complements are metacylic. Hence the result follows combining Proposition~\ref{CaminaGroups} and Proposition~\ref{FrobeniusXAbeliano}. 
\end{proof}

\begin{proposition} \label{Hcoprime}
        Let $G = F \times H$ where $F$ is a Frobenius group with Frobenius complement $C$ and $H$ is a finite group of coprime order to $F$.
        Suppose that the first Zassenhaus conjecture holds for $C \times H$. 
        Then it also holds for $G$.
\end{proposition}
\begin{proof}
As in the proof of Proposition~\ref{FrobeniusXAbeliano} we can deduce that the first Zassenhaus conjecture also holds for $C$ and $H$ by Remark~\ref{rem:ZC_for_direct_factors}.
Let $u$ be a torsion element of $\V(\ZZ G)$ and denote by $N$ the Frobenius kernel of $F$.

Assume first $\gcd(o(u), |N|) = 1$. Then projecting $u$ via the homomorphism $\ZZ G \rightarrow \ZZ G/N$ and using that the first Zassenhaus conjecture holds for $G/N=C\times H$ by assumption we can conclude that $u$ is rationally conjugate to an element of $C \times H$ by \cite[Theorem~2.7]{DoJu}.

So assume that $\gcd(o(u), |N|) \not= 1$. We will prove that all partial augmentations of $u$ but one vanish and  the result will follow by \ref{knownfacts}.(\ref{MRSW}).
Let $p$ be a common prime divisor of $o(u)$ and $|N|$.
Then $p\nmid |H|$ and by Theorem~\ref{HertweckZ_p} we deduce that $u_p\sim_{\ZZ_p G} x_p\in N_p$. 
Repeating this process for every prime $p$ dividing $o(u)$ and also $|N|$ we may define $x=\prod x_p \in N$ and $\pi =\prod p$. 
Now using Proposition~\ref{HertweckPaugmentation} we have that $\varepsilon_g(u)=0$ for every $g\in G$ with $g_{\pi} \not\sim x$. 
Furthermore, by the assumptions on the order of $G$ and that $F$ is a Frobenius group, we have that 
$\varepsilon_g(u)=0$ for every $g\in G$ with a common prime divisor of $o(g)$ and $|C|$. 
Therefore the distinct non-vanishing partial augmentations of $u$ are of the form $\varepsilon_{xh}(u)$ for $h\in H$. 
It follows that the image $\bar{u}$ of $u$ under the homomorphism $\ZZ G\rightarrow \ZZ \bar{G}=\ZZ (G/F)$ has distinct non-vanishing partial augmentations $\varepsilon_{\bar{h}}(\bar{u})$ for $h\in H$. Hence precisely one partial augmentation of $u$ does not vanish, since the first Zassenhaus conjecture holds for $H$. This finishes the proof. 
\end{proof}

Now we look at Frobenius groups and the general Bovdi problem. 

\begin{proposition} \label{Fcomplement} Let $C$ be a finite group isomorphic to a Frobenius complement. Then (Gen-BP) holds for $C$. 
\end{proposition}

\begin{proof}

 By \cite[\S 18]{Passman68} the structure of a Frobenius
  complement $C$ is as follows. By $H = \boxxedcomp{Q}{N}$ we indicate that $H$ has a normal subgroup $N$ such that $H/N \simeq Q$.
 Denote by $F = \operatorname{F}(C)$ the Fitting subgroup of $C$. 

\begin{enumerate}
\item[(1)] If the Sylow $2$-subgroup of $C$ is cyclic then $C$ is a Z-group.  

\item[(2a)] Suppose that $F_2$ is cyclic. Then $C$ is metabelian.

\item[(2b)] Suppose that $F_2 \cong Q_8$. Then    
$$ (i) \ \ C = \boxxedcomp{C_2}{\SL (2,3) \times M} 
  \ \ \mbox{or}  \ \ (ii) \ \ C = \boxedcomp{ \SL (2,3) \times M}  \ \ 
   \mbox{or} \ \ (iii) \ \ C = \boxedcomp{Q_8  \times M}\ ,$$   
where $M$ is a  metacyclic Z-group of odd order coprime to the order of $\SL (2,3)$ and $Q_8$ respectively. 

\item[(2c)] Suppose that $F_2 \cong Q_{2^n}$ with $n \geq 4$.
Then   
$$  \ \ C = \boxxedcomp{C_2}{C_{2^{n-1}} \times M} \ \ $$ 
where $M$ is a
  metacyclic Z-group of odd order and the Sylow $2$-subgroup of $C$ is isomorphic to $Q_{2^n}$. So $C
  \cong F_2 \times M$.  
\item[(3)] If $C$ is non-solvable then 
$$ (i) \ \ C = \boxxedcomp{C_2}{\SL (2,5) \times M} \ \ \mbox{or}  \ \ (ii) \ \ C = \boxedcomp{\SL (2,5) \times M}\ ,$$ 
where $M$ is a
  metacyclic Z-group of odd order coprime to the order of $\SL (2,5)$.  
\end{enumerate}
\vskip1em
All Sylow subgroups of $M$ are cyclic. Thus $M$ has in each case a Sylow tower. 
It follows that $C$ is metabelian or that $C$ has a normal Hall subgroup $N$
with a Sylow tower, $C/N$ is of order coprime to $|N|$ and isomorphic to  
$$ \boxxedcomp{C_2}{\SL (2,3)}\ , \qquad \boxedcomp{\SL(2,5)}\ , \qquad \boxxedcomp{C_2}{\SL (2,5)}\ ,$$
or trivial. We have that (ZC1) is valid for all these possible complements of $N$ by \cite{HoeKim,BovHer,DJPM}. 
Thus by Corollary~\ref{prop:corollary_sylowtower} we get that (Gen-BP)
holds for $ C .$ 
Finally, if $C$ is metabelian, then (Gen-BP) holds by \cite[Corollary 1.4]{DokSeh}.    
\end{proof} 

\begin{corollary} \label{Frobenius} Let $G$ be a  Frobenius group. 
Then (Gen-BP) holds for $G.$ 
\end{corollary} 

\begin{proof} Let $N$ be the Frobenius kernel of $G $ and $C$ a
  Frobenius complement. By
  Thompson's famous result $N$ is a nilpotent Hall subgroup. Thus 
we may apply Proposition \ref{normalSylowp} and get that (Gen-BP)
holds for $G$ if, and only if, it holds for $C .$  But this follows
  from Proposition \ref{Fcomplement}.  
\end{proof} 

\begin{remark} It is unknown whether (ZC1) holds for all Frobenius
  groups.
The arguments in the proofs used above show that for many of them this
is indeed the case.    
\end{remark}

\begin{remark}\label{remark_on_Frobenius} Let $G$ be a Frobenius group with Frobenius kernel $N$,
  $C$ a Frobenius complement of $G$
  and $H$ a finite group for which (Gen-BP) (and therefore (KP))
  holds. The analogous question to Propositions
  \ref{FrobeniusXAbeliano} and \ref{Hcoprime} is whether (Gen-BP)
  holds for $G \times H$. The analogue to \ref{Hcoprime} follows
  immediately from Proposition~\ref{normalSylowp} provided (Gen-BP) is valid
  for $C \times H$ and $\gcd(|H|,|G|) = 1$. 
 
If $H$ is nilpotent then we may write $H = H_1 \times H_2$ such that
  $\gcd(|H_1|,|C|) = 1 $ and each prime dividing $|H_2|$ divides $|C|.$
  Then $N \times H_1$ is a normal nilpotent Hall subgroup of
  $G \times H.$ From Proposition~\ref{normalSylowp} and
  Proposition~\ref{Fcomplement} we see that (Gen-BP) holds for $G
  \times H_1 .$ Moreover (Gen-BP) holds for $G \times H$ if it holds
  for $C \times H_2 .$  
 The latter is true when $C \times H_2$ has a Sylow tower and 
this is the case when  
  $C $ has a Sylow tower. 
  
  In case when $H$ is abelian we come back to this 
  at the end of the next section.             
\end{remark}

\section{Extending Coefficients}\label{sec:HELP}

Let $G$ be a finite group. 
There is an obvious generalization of the first Zassenhaus conjecture to group rings where the coefficients are allowed to come from rings of algebraic integers.

\begin{problem}\label{ZC1_OG} Let $\O$ be the ring of algebraic integers in a number field $K$. Let $u \in \V(\O G)$ be a unit of finite order. Is $u$ conjugate by a unit of $KG$ to an element of $G$?
\end{problem}

This question is connected to certain instances of the classical first Zassenhaus conjecture, where the coefficients come from $\ZZ$, as can be seen from the following result of Hertweck (see \cite[Proposition~8.2]{Hertweck08}).
        
\begin{proposition}\label{prop:Hertweck_reduction} Let $G$ be a finite
  group and let $A$ be a finite abelian group of exponent $m$. 
        Suppose that any torsion element of $\V(\ZZ[\zeta_m]G)$ is conjugate in $(\QQ(\zeta_m)G)^\times$ to an element of $G$. Then (ZC1) holds for $G \times A$. \end{proposition}

Hertweck's proof of Proposition~\ref{prop:Hertweck_reduction} actually
proves part a) of the following Proposition. It may be easily modified for (Gen-BP)
-- stated  as part b).  Of course we say that (Gen-BP) holds for
$\ZZ[\zeta_m]G$ if and only if the unique non-vanishing trace of a torsion element
$u \in \V(\ZZ[\zeta_m]G)$ is $\varepsilon_{G[o(u)]}(u)$.

\begin{proposition}\label{extendedHertweck_reduction} 
Let $G$ be a finite group, let $A$ be a finite abelian group of
exponent $m $ and let $t \in \mathbb{N}$ such that $t$ divides the
exponent of $G \times A.$   
\begin{itemize}
\item[a)]
Suppose that any  $u\in \V(\ZZ[\zeta_m]G)$ of order dividing $t$ 
is conjugate in
  $(\QQ(\zeta_m)G)^\times$ to an element of $G$. Then each torsion unit of
  order $t$ of $\V(\ZZ [G \times A])$ is rationally conjugate to an element of $G
  \times A$. 
\item[b)]         
If (Gen-BP) holds for
$\ZZ[\zeta_m]G$ then (Gen-BP) holds for $\ZZ [G \times A]$. \\
More precisely if for any $u \in \V(\ZZ[\zeta_m]G)$ of order
dividing $t$ we have that 
$\varepsilon_{G[k]}(u) \neq 0 $ if, and only if, $k = o(u) $,  then  
for any $u \in \V(\ZZ (G \times A))$ of order
$t$ we also have that 
$\varepsilon_{(G \times A)[k]}(u) \neq 0 $ if, and only if, $k = t .$
\end{itemize}
\end{proposition}

Let $\O$ be the ring of algebraic integers in a number field $K$. Then $\O$ is $G$-adapted (i.e.\ it is an integral domain of characteristic $0$ and no prime divisor of the order of $G$ is invertible in $\O$).
Many of the usual theorems on torsion units of integral group rings
still hold in the case of coefficients from a $G$-adapted ring (for the first two statements see \cite[Theorem~1.1]{Hertw06} and for the last two \cite[Proposition~2.2, Theorem~2.1]{HertweckPa}).

\begin{nr}\label{lemma:standardfacts} Let $R$ be a $G$-adapted ring and let $u \in \V(R G)$ be an element of order $n$. Let $K$ be the field of fractions of $R$. Then the following statements hold:
\begin{enumerate}
\item $n$ divides the exponent of $G$. 
\item $\varepsilon_1(u) = 0$ if $n \not= 1$.
\item \label{vanishing_pas} If $x\in G$ then $\varepsilon_x(u) = 0$ whenever $o(x) \nmid n$.
\item \label{conjugacy} $u$ is conjugate by a unit of $KG$ to an element of $G$ if and only if, for all divisors $d$ of $n$, all partial augmentations of $u^d$ but one vanish. 
\end{enumerate}
\end{nr}

In view of the last statement, it is desirable to have tools at hand that can be used to produce constraints on partial augmentations of torsion units of $\V(R G)$. In the case of coefficients from $\mathbb{Z}$ this can be achieved, for example, by the well-known HeLP method (see Section~\ref{sectionHeLPMethod}).  We present an extension of this method to rings of algebraic integers.

In the sequel we fix a finite group $G$, a ring of algebraic integers $\O$, an element $u \in \V(\O G)$ of order $n$ and augmentation 1 and a complex primitive $n$th root of unit $\zeta$.
We can linearly extend each ordinary ($p$-Brauer) character $\chi$ of $G$ to a character of $\V(\O G)$ (of the $p$-regular torsion elements of $\V(\O G)$). With exactly the same proof (cf.\ e.g.\ \cite[§~4]{HertweckPa}) the formula for the multiplicity of roots of unity as eigenvalues remains valid. 
The multiplicity $\mu_\ell(u, \chi)$ of $\zeta^\ell$ as eigenvalue of $D(u)$, where $D$ is a representation affording $\chi$, is

\[ \mu_\ell(u, \chi) = \frac{1}{n} \sum_{d \mid n} \Tr_{\QQ(\zeta^d)/\QQ}\left(\chi(u^d)\zeta^{-d\ell}\right). \]
Clearly, this expression has to be a non-negative integer. Note that we have $\chi(u^d) \in \QQ(\zeta^d)$, as this is the sum of all the eigenvalues of $D(u^d)$. We can isolate the term for $d = 1$, 
\begin{equation} \frac{1}{n} \left( \Tr_{\QQ(\zeta)/\QQ}\left(\chi(u)\zeta^{-\ell}\right) +  \sum_{1 \not= d \mid n} \Tr_{\QQ(\zeta^d)/\QQ}\left(\chi(u^d)\zeta^{-d\ell}\right) \right) = \mu_\ell(u, \chi) \in \mathbb{Z}_{\geq 0}, \label{eq:formulat_mu}\end{equation}
and assume by induction that the latter sum is known. We have $\chi(u) = \sum_{x^G} \varepsilon_x(u) \chi(x)$ (see \cite[Theorem~3.2]{HertweckPa} for Brauer characters) and \ref{lemma:standardfacts}.\eqref{vanishing_pas} guarantees $\varepsilon_x(u) = 0$ whenever $o(x) \nmid n$. So all the character values at conjugacy classes which might have a non-zero partial augmentation are contained in $\QQ(\zeta)$.  That also the partial augmentations are contained in $\QQ(\zeta)$, so that we can use the $\QQ$-linearity of the trace to simplify further, is guaranteed by the following lemma. 

\begin{lemma}\label{lemma:pas_in_Zzeta} Let $\O$ be a ring of algebraic integers and $G$ be a finite group. If $u \in \V(\O G)$ is an element of order $n$ and $\zeta$ a primitive $n$th root of unity, then $\varepsilon_x(u) \in \ZZ[\zeta] \cap \O$ for every $x \in G$. \end{lemma}
\begin{proof} Let $\Irr(G) = \{\chi_1, ..., \chi_h\}$ be the set of irreducible characters of $G$ and let $\{x_1, ..., x_h\}$ be a set of class representatives of $G$. Without loss of generality we may assume that $x_1, ..., x_d$ are the class representatives with an order a divisor of $n$. Then, by \ref{lemma:standardfacts}.\eqref{vanishing_pas}, $\varepsilon_{x_j}(u) = 0$ for $j \in \{d+1, .., h\}$ and it remains to show that $\varepsilon_{x_j}(u) \in \ZZ[\zeta]$ for $j \in \{1, .., d\}$. We have \[\begin{pmatrix} \chi_1(u) \\ \chi_2(u) \\ \vdots \\ \chi_h(u) \end{pmatrix} = \begin{pmatrix} \chi_1(x_1) & \chi_1(x_2) & \hdots & \chi_1(x_h) \\ \chi_2(x_1) &  \chi_2(x_2) & \hdots & \chi_2(x_h) \\ \vdots & \vdots & \ddots & \vdots \\ \chi_h(x_1) &  \chi_h(x_2) & \hdots & \chi_h(x_h) \end{pmatrix}\begin{pmatrix} \varepsilon_{x_1}(u) \\ \varepsilon_{x_2}(u) \\ \vdots \\ \varepsilon_{x_h}(u) \end{pmatrix}. \] As $u$ is of order $n$, the column on the left hand side is an element of $\ZZ[\zeta]^h$.
Denote the character table of $G$ with the ordering above as $C = (c_{i,j})$. Let $J = \{1, ..., d\}$. As $C$ is invertible, we can choose $I \subseteq \{1, ..., h\}$ with $|I| = d$ in such a way that the $d\times d$-submatrix $(c_{i,j})_{i \in I, j\in J}$ is invertible. 
Note that $(c_{i,j})_{i \in I, j\in J} \in \GL(d, \QQ(\zeta))$, as the entries are character values of elements with an order a divisor of $n$. Hence \[ (\varepsilon_{x_j}(u))_{j \in \{1, .., d\}} = (c_{i,j})_{i \in I, j\in J}^{-1}(\chi_i(u))_{i \in I} \in \QQ(\zeta)^d.\] As these partial augmentations are algebraic integers, we infer $\varepsilon_{x}(u) \in \ZZ[\zeta] \cap \O$ for all $x \in G$. This completes the proof.
\end{proof}
\begin{remark} Although the partial augmentations of normalized units in $\O G$  of order $n$ are contained in $\ZZ[\zeta_n]$, this is in general not true for the coefficients: Let $ G = S_3$ and $\O = \ZZ[\zeta_9]$. Then \begin{align*} u &= (1,2,3)+\zeta_9(1,2)+\zeta_9^4(2,3)+\zeta_9^7(1,3)\\ &= (1,2,3)+\zeta_9((1,2)+\zeta_3(2,3)+\zeta_3^2(1,3))  \in \V(\O G)\end{align*} is a unit of order $3$.\end{remark}

We can choose a basis $B$ of $\ZZ[\zeta] \cap \O$ over $\ZZ$ and express $\varepsilon_x(u) = \sum_{b \in B}\alpha_{x, b}b$ with $\alpha_{x, b} \in \ZZ$. Then 
\[ \Tr_{\QQ(\zeta)/\QQ}\left(\chi(u)\zeta^{-\ell}\right) = \sum_{x^G}\sum_{b \in B}\alpha_{x, b}\Tr_{\QQ(\zeta)/\QQ}\left(\chi(x)\zeta^{-\ell}b\right).\]
So using \eqref{eq:formulat_mu}, we get a system of linear inequalities over $\ZZ$:
\begin{equation} \sum_{x^G}\sum_{b \in B}\alpha_{x, b}\Tr_{\QQ(\zeta)/\QQ}\left(\chi(x)\zeta^{-\ell}b\right) +  \sum_{1 \not= d \mid n} \Tr_{\QQ(\zeta^d)/\QQ}\left(\chi(u^d)\zeta^{-d\ell}\right)  = n\mu_\ell(u, \chi) \in n\mathbb{Z}_{\geq 0}. \label{eq:system} \end{equation}
Note that compared to the ``classical'' HeLP method, where $\O = \ZZ$, the number of variables grows by a factor $[K\cap\QQ(\zeta):\QQ]$. Now we want to exploit Proposition \ref{prop:Hertweck_reduction} to verify the first Zassenhaus conjecture for direct products $G \times A$, where $A$ is an arbitrary finite abelian group.

For a given divisor $n$ of the exponent of $G$ and $\zeta$ an arbitrary complex root of unity, we will show that each element of order $n$ of $\V(\ZZ[\zeta]G)$ (if it exists) is conjugate in $(\QQ(\zeta)G)^\times$ to an element of $G$ by showing that every solution of (\ref{eq:system}) is in accordance with the condition of \ref{lemma:standardfacts}.\eqref{conjugacy} or that there is no solution to (\ref{eq:system}) at all (in case there is no group element of that order). We can again employ Lemma \ref{lemma:pas_in_Zzeta} to see that it is enough to do this for $\zeta$ a primitive $n$th root of unity. So for each group $G$ we are left with the problem of finding the solutions to a finite number of systems of linear inequalities over the integers. We will choose $B = \{1, \zeta, ..., \zeta^{\varphi(n) - 1}\}$ as basis of $\ZZ[\zeta]$ over $\ZZ$ where $\varphi$ denotes the Euler totient function. 

A rational prime $p$ is called \emph{totally ramified} in an algebraic number field $K$ (or rather its ring of algebraic integers $\O$) if for each prime ideal $\mathfrak{p}$ containing the ideal $p\O$, the field $\O/\mathfrak{p}$ has cardinality $p$. For example, $p$ is totally ramified in $\ZZ[\zeta_{p^a}]$ for $a \in \ZZ_{\geq 0}$, see e.g.\ \cite[Proposition~7.4.1]{WeissAN}. Based on a result of Cohn-Livingstone \cite{CohLiv}, one can establish extra constraints for torsion units of $\ZZ G$, sometimes called the ``Wagner test'', cf.\ \cite[Proposition~3.1]{HeLPOverview}. With an adapted proof we get the following version for coefficients in rings of algebraic integers.

\begin{proposition}[``Wagner test'']\label{prop:Wagner} Let $G$ be a finite group, $p$ be a prime and $\O$ a ring of integers such that $p$ is totally ramified in $\O$. Let $u \in \V(\O G)$ be a unit of order $o(u) = p^j m$ with $m \not= 1$. Then for $s \in G$ and $\mathfrak{p}$ a prime ideal containing $p\O$, we have $$\smashoperator[r]{\sum\limits_{x^G,\ x^{p^j} \sim s}} \varepsilon_x(u) \equiv \varepsilon_{s}(u^{p^j}) \mod \mathfrak{p}.$$ 
\end{proposition}

\begin{proof} Let $u = \sum_{g\in G} u_g g \in\V(\O G),$ set $q = p^j$ and $v = u^q$. By definition
\begin{equation} \varepsilon_s(v) = \sum_{\substack{(g_1, ..., g_q) \in G^q \\ g_1  ...  g_q \sim s}} \prod_{j=1}^q u_{g_j}. \label{vsum} \end{equation}
The set over which the sum is taken can be decomposed into $\mathcal{M} = \{(g, ..., g) \in G^q: g^q \sim s \}$ and $\mathcal{N} = \{ {(g_1, ..., g_q) \in G^q}:  g_1  ...  g_q \sim s \; \text{and} \; \exists \, r, r' : g_r \not= g_{r'}\}$. 

The cyclic group $C_q = \langle t \rangle$ of order $q$ acts on the set $\mathcal{N}$ by letting the generator $t$ shift the entries of a tuple to the left, i.e., $(g_1, g_2, g_3, ..., g_q) \cdot t = (g_2, g_3, ..., g_q, g_1)$. Note that all orbits have length $p^i$ with $i \geq 1$.  For elements in the same orbit, the same integer is summed up in \eqref{vsum}.
Using that $\O/\mathfrak{p}$ has characteristic $p$ and that $\operatorname{U}(\O/\mathfrak{p})$ is a group of order $p - 1$ we get
\[\varepsilon_s(v) =  \smashoperator[r]{\sum_{(g, ..., g) \in \mathcal{M}}} u_g^q + \sum_{(g_1, ..., g_q) \in \mathcal{N}} \prod_{j=1}^q u_{g_j} \equiv \smashoperator[r]{\sum_{(g, ..., g) \in \mathcal{M}}} u_g^q \equiv \smashoperator[r]{\sum_{(g, ..., g) \in \mathcal{M}}} u_g \equiv \smashoperator[r]{\sum_{x^G,\ x^{p^j} \sim s}} \varepsilon_x(u) \mod \mathfrak{p}. \qedhere \] 
\end{proof}

Using the constraints we obtained so far in this section will be called the \emph{extended HeLP method}, abbreviated as \emph{HELP} method. This has been implemented in the computer algebra system \textsf{GAP} \cite{GAP4} and applied to groups of small order.

If $G$ is nilpotent, a Camina group, a cyclic-by-abelian group or if
it has a normal Sylow $p$-subgroup with abelian quotient, then the
first Zassenhaus conjecture is known for $G \times A$, for $A$ a
finite abelian group. So Proposition \ref{prop:Hertweck_reduction}
together with the above method will not provide us in these case with anything new. If we filter all groups up to order $95$ that are not covered by what is said before, we are left with $17$ groups (up to order $100$, there are $73$ such groups).

We will list in Table \ref{tab:HELP} the \textsc{SmallGroup ID}s of those groups of order at most $95$ that are not covered by the previously mentioned results together with the structure description in the first two columns. The third column contains the orders $n$ of units in $\V(\ZZ[\zeta_n]G)$, where the HELP method (including the ``Wagner test'') does not provide a complete solution; in parentheses the number of distributions of non-trivial partial augmentations that cannot be ruled out is indicated. In case there are only trivial partial augmentations left, a checkmark is included. The last column contains the orders where either the Wagner test or the so-called ``quotient method'' (a unit would map to a unit with an already eliminated distribution of partial augmentations in an integral group ring of a quotient group) can be used together with the number for such distributions where this applies; if the quotient method does not provide new information, the zero is omitted. 

\begin{table}[h!]
\caption{Groups of order at most $95$ investigated with the HELP method.}\label{tab:HELP}
\begin{center}
\begin{tabular}{cccc}\hline
\textsc{SmallGroupID} & Structure Description & Order & Wagner test / quotient method \\ \hline\hline
\texttt{[24,12]} & $S_4$ & 4(4) & 4(4) \\
\texttt{[48,28]} & $C_2 . S_4 = \SL(2,3) . C_2$ & 8(8)\\
\texttt{[48,29]} & $\GL(2,3)$ & 8(4) & 4(1), 8(4) \\
\texttt{[48,30]} & $A_4 : C_4$  & 4(8) & 4(21 / 5)\\
\texttt{[48,48]} & $C_2 \times S_4$ & 4(16) & 4(12) \\
\texttt{[60,5]} & $A_5$ & 6(2)\\
\texttt{[72,15]} & $((C_2 \times C_2) : C_9) : C_2$ & 4(4) & 4(4), 12(3) \\
\texttt{[72,22]} & $(C_6 \times S_3) : C_2$  & \checkmark & 4(2)\\
\texttt{[72,23]} & $(C_6 \times S_3) : C_2$ & \checkmark & 4(2)\\
\texttt{[72,24]} & $(C_3 \times C_3) : Q_8$  & \checkmark \\
\texttt{[72,31]} & $(C_3 \times C_3) : Q_8$ & \checkmark \\
\texttt{[72,33]} & $(C_{12} \times C_3) : C_2$ & \checkmark \\
\texttt{[72,35]} & $(C_6 \times C_6) : C_2$: & \checkmark & 4(2)\\
\texttt{[72,40]} & $(S_3 \times S_3) : C_2$ = $S_3 \wr S_2$ & 3(2), 6(4) & 4(2) \\
\texttt{[72,42]} & $C_3 \times S_4$ & 4(4), 12(8) & 4(4)\\
\texttt{[72,43]} & $(C_3 \times A_4) : C_2$ & 4(4) & 4(4), 12(2)\\
\texttt{[72,44]} & $A_4 \times S_3$  & \checkmark \\ \hline
\end{tabular}
\end{center}
\end{table}

Note that there are problems with normalized units of order a power of
$2$ if and only if $G$ maps onto $S_4$. In particular those
distributions of partial augmentations that cannot be ruled out in
these groups always map on one of the distributions of partial
augmentations in $S_4$ that cannot be ruled out. We will discuss this
below. So, in order to solve Problem \ref{ZC1_OG} for all groups of order
at most $95$, one has only to deal with $S_4$, $A_5$ and the
wreath product $S_3 \wr S_2$. This will also be discussed later, first we
record the following consequence of these calculations and Proposition
\ref{extendedHertweck_reduction}.

\begin{proposition}\label{ExtendedResult}
        Let $G$ be a group of order at most $95$ and $A$ a finite abelian group. Then the first Zassenhaus conjecture holds for $G \times A$ except if $G$ maps onto $S_4$ or $G \simeq A_5$ or $G \simeq S_3 \wr S_2$. \\
        If $G\simeq A_5$ or $G\simeq S_3 \wr S_2$, then the first
        Zassenhaus conjecture holds for $G \times A$ if $A$ is a
        $3'$-group. If $G$ maps onto $S_4$, then the first Zassenhaus
        conjecture holds for $G \times A$ if $4$ does not divide the
        exponent of $A$. \\
In the case that $S_4$ is an image of $G$ all units of $V(\ZZ (A \times
        G))$ whose order is not divisible by $4$  are rationally conjugate
        to an element of $A \times G.$ 
\end{proposition}

We will now focus on one of the problematic partial augmentations for $S_4$ that could not be ruled out yet. Note that \texttt{CharacterTable("S4")} in \textsf{GAP} produces a permutation of the columns of the character table \texttt{CharacterTable(SmallGroup(24,12))}. We will use the notation for conjugacy classes of the latter table, i.e.\ $2a$ contains the transpositions $(\bullet\bullet)$ and $2b$ the double transpositions $(\bullet\bullet)(\bullet\bullet)$. The irreducible characters of $S_4$ will be denoted by $\chi_{1a} = 1$, $\chi_{1b} = \sgn$, $\chi_2$ (this is the inflation of the irreducible non-linear character of $S_3$), $\chi_{3a}$ and $\chi_{3b} = \chi_{3a} \otimes \sgn$. The character table of $S_4$ we are using is thus as follows (dots indicate zeros):

\begin{center}
\begin{tabular}{cccccc} \hline\hline
class & 1a & 2a & 3a & 2b & 4a \\
cycletype & () & $(\bullet\bullet)$ & $(\bullet\bullet\bullet)$ & $(\bullet\bullet)(\bullet\bullet)$ & $(\bullet\bullet\bullet\, \bullet)$ \\ \hline
$\chi_{1a}$ & 1 & 1 & 1 & 1 & 1 \\
$\chi_{1b}$ & 1 & $-1$ & 1 & 1 & $-1$ \\
$\chi_{2}$ & 2 & . & $-1$ & 2 & . \\
$\chi_{3a}$ & 3 & $-1$ & . & $-1$ & 1 \\
$\chi_{3b}$ & 3 & 1 & . & $-1$ & $-1$ \\ \hline\hline
\end{tabular}
\end{center}

\vskip1em

For normalized units $u$ of order $4$ we always have $u^2 \sim 2b$ and we are left with the following four cases of distributions of partial augmentations that do not correspond to units that are conjugate in $\QQ(i) S_4$ to an element of the group:

\vskip1em

\noindent\begin{minipage}{.4\textwidth}
\underline{Case 1}:\\
$\varepsilon_{2a}(u) = i$\\
$\varepsilon_{2b}(u) = 1$\\
$\varepsilon_{4a}(u) = -i$\\

\underline{Case 2}:\\
$\varepsilon_{2a}(u) = 1+i$\\
$\varepsilon_{4a}(u) = -i$\\
\end{minipage}
\begin{minipage}{.4\textwidth}
\underline{Case 3}:\\
$\varepsilon_{2a}(u) = -i$\\
$\varepsilon_{2b}(u) = 1$\\
$\varepsilon_{4a}(u) = i$\\

\underline{Case 4}:\\
$\varepsilon_{2a}(u) = 1-i$\\
$\varepsilon_{4a}(u) = i$\\
\end{minipage}\\
Of course, all partial augmentations not recorded are zero.

Consider the ring homomorphism $\tau \colon \ZZ[i]S_4 \to \ZZ[i]S_4$ induced by complex conjugation on the coefficients. Then one can see that a unit as in case 1 exists if and only if a unit as in case 3 exists and similarly for case 2 and 4. So it suffices to consider the first 2 cases.

Assume we are in case 1, i.e.\ $u = \sum u_g g \in \V(\ZZ[i]S_4)$ is of order $4$, $u^2 \sim 2b = (\bullet\bullet)(\bullet\bullet)$, a double transposition, and \[\varepsilon_{2a}(u) = i, \ \varepsilon_{2b}(u) = 1, \ \varepsilon_{4a}(u) = -i.\] Clearly $\chi_{1a}(u) = 1$, $\chi_{1b}(u) = 1$, $\chi_2(u) = 2$, $\chi_{3a}(u) = -1 - 2i$ and $\chi_{3b}(u) = -1 + 2i$. Let $D$ be the direct sum of the representations corresponding to $\chi_{1a}, \chi_{1b}, \chi_2, \chi_{3a}$ and $ \chi_{3b}$ (in this order). Then in a diagonalized form $D(u)$ looks as follows:
\begin{equation} D(u) \sim \left(1, 1, \left(\begin{smallmatrix} 1 & \\ & 1 \end{smallmatrix}\right), \left(\begin{smallmatrix} -1 & & \\ & -i & \\ & & -i \end{smallmatrix}\right), \left(\begin{smallmatrix} -1 & & \\ & i & \\ & & i \end{smallmatrix}\right) \right). \label{eq:image_of_u_under_D}\end{equation}
From this it immediately follows that $u$ is in the kernel of the natural homomorphism $\V(\ZZ[i]S_4) \to \V(\ZZ[i]S_3)$. In \cite[Section~2]{LT91}, Luthar and Trama obtained certain congruences modulo $|G|$ from the integrality of the coefficients of the group rings elements. In their paper it turned out to be sufficient to exclude the existence of certain units of order $4$ and $6$ in $\V(\ZZ S_5)$ and they could conclude that the first Zassenhaus conjecture holds for $S_5$. We can obtain similar restrictions (modulo the ideal $24\ZZ[i]$) that express that the coefficients of the units in question actually lie in $\ZZ[i]$. However in this case, these systems do not provide us with contradictions. There are even solutions modulo $6\ZZ[i]$ corresponding to matrices of order $4$. These solutions correspond to normalized units of order $4$ in $\ZZ[i, \frac{1}{2}]S_4$, which even lie in an order of $\QQ(i)S_4$ containing $\ZZ[i]S_4$. One such example is
\begin{align*} u =   \frac{1}{4}  & \Big( (-1+i)(1, 2) + (1+i)(1, 3) + i (1,4)+ i(2, 3) -(2, 4) + (3,4) +(1, 2, 3)  \\  &  + (-1+i)(1,3,4) - (1 + i)(1,4,2) + (2,4,3) +(2-i)(1,2)(3,4) + (2+i)(1,3)(2,4) \\ & - (1,2,3,4) - i(1,2,4,3) - (1+i)(1,3,2,4) + (1,4,2,3) + (1-i)(1,4,3,2) -  i(1,3,4,2) \Big), \end{align*}
which is of order $4$.

Note that a torsion unit of $\ZZ[i]S_4$ in case 2 above is not in the kernel of the natural homomorphism $\V(\ZZ[i]S_4) \to \V(\ZZ[i]S_3)$, but rather maps to an involution.

The other two groups of order at most $95$ that cannot be handled and do not project onto $S_4$ are $A_5$ and $S_3 \wr S_2$. We also provide  all remaining non-trivial distributions of partial augmentations for these groups. In the sequel denote by $\zeta =\zeta_3$ a primitive $3$rd root of unity.

For the group $A_5$, units $u$ with augmentation one and order $6$ in $\ZZ[\zeta]A_5$ can not be proved to be conjugate within $\QQ(\zeta)A_5$ to a group element using the HELP method. Let $2a$ and $3a$ denote the unique $A_5$-conjugacy class of involutions and elements of order $3$, respectively. In all cases that cannot be ruled out, $u^3 \sim 2a$, $u^2 \sim 3a$ and
\vskip1em
\noindent\begin{minipage}{.4\textwidth}
\underline{Case 1}:\\
$\varepsilon_{2a}(u) = -2\zeta$\\
$\varepsilon_{3a}(u) = 1 + 2\zeta$\\
\end{minipage}
\begin{minipage}{.4\textwidth}
\underline{Case 2}:\\
$\varepsilon_{2a}(u) = -2\zeta^2$\\
$\varepsilon_{3a}(u) = 1 + 2\zeta^2$.\\
\end{minipage}\\
These cases can also not be ruled out by using Brauer characters (which might provide additional information in case of non-solvable groups) or the so-called lattice method \cite{BMM10}.  

For $G = S_3 \wr S_2$, units of order $3$ and $6$ with non-trivial partial augmentations in $\ZZ[\zeta]G$ remain after the application of the HELP method. 
For elements of order $3$ the non-trivial distributions of partial augmentations are
\[ (\varepsilon_{3a}(u), \varepsilon_{3b}(u)) \in \left\{\left(-\zeta, -\zeta^2\right), \left(-\zeta^2, -\zeta\right)  \right\}. \]
For elements of order $6$ with non-trivial partial augmentations that cannot be ruled out with HELP we always have $u^3 \sim 2c$ (the class of involutions in $S_2$) and

\noindent\begin{minipage}{.4\textwidth}
\underline{Case 1}:\\
$u^2 \sim 3b$\\
$\varepsilon_{2b}(u) = 1$\\
$\varepsilon_{2c}(u) = 1$\\
$\varepsilon_{6b}(u) = -1$\\

\underline{Case 2}:\\
$u^2 \sim 3b$\\
$\varepsilon_{2b}(u) = -1$\\
$\varepsilon_{2c}(u) = 1$\\
$\varepsilon_{6b}(u) = 1$\\
\end{minipage}
\begin{minipage}{.4\textwidth}
\underline{Case 3}:\\
$u^2 \sim 3a$\\
$\varepsilon_{2a}(u) = 1$\\
$\varepsilon_{2c}(u) = 1$\\
$\varepsilon_{6a}(u) = -1$\\

\underline{Case 4}:\\
$u^2 \sim 3a$\\
$\varepsilon_{2a}(u) = -1$\\
$\varepsilon_{2c}(u) = 1$\\
$\varepsilon_{6a}(u) = 1$\\
\end{minipage}\\
Note that case 1 and 3 and case 2 and 4 lie in the same $\operatorname{Aut}(S_3 \wr S_2)$ orbit (interchanging the two factors isomorphic to $S_3$ in the base group).

\begin{remark}\label{remarkSL25_S5_2S5} The HELP method can successfully be applied to the unique perfect group of order $120$, $\SL(2, 5)$. This proves the first Zassenhaus conjecture for $\SL(2, 5) \times A$, $A$ a finite abelian group.
\end{remark}

\begin{corollary} \label{ZC95Frob}
        (ZC1) holds for $G\times A$ where $A$ is any finite abelian
        group and $G$ is a Frobenius group whose complement $C$ either has
        order at most $95$ and is not isomorphic to $C_2.S_4$ or $C = \SL(2,5)$. 
\end{corollary}
\begin{proof}
        By \cite[Theorem 18.1]{Passman68} we know that all Sylow
        $p$-subgroups of Frobenius complements are cyclic (in case $p$
        is odd) or cyclic or quaternion (for $p = 2$). However the
        groups of order at most $95$ in Table~\ref{tab:HELP} that
        cannot be handled have at least one Sylow subgroup which is
        not of that form except the case of $C_2.S_4$.   
The result now follows combining Proposition~\ref{FrobeniusXAbeliano}, Proposition~\ref{ExtendedResult} and Remark~\ref{remarkSL25_S5_2S5}.
\end{proof}

\begin{remark} \label{s5}
        For $S_5$ the HELP method can successfully be applied except for units of order $4$, $6$ and $12$. In these cases the problematic partial augmentations are as follows. Let $2a$ be the conjugacy class of involutions contained in $A_5$. For partial augmentations of $u \in \V(\ZZ[i]S_5)$ that cannot be ruled out we have $u^2 \sim 2b$ and
        \[((\varepsilon_{2a}(u), \varepsilon_{2b}(u), \varepsilon_{4a}(u)) \in \{ (0, 1-i, i), (1, -i, i), (0, -i, 1+i), (0, i, 1-i), (0, 1+i, -i), (1, i, -i)  \}. \]
        For elements $u \in \V(\ZZ[\zeta_3]S_5)$ of order $6$ the following remain (always $u^2 \sim 3a$): 
        \begin{align*} u^3 \sim 2b \ \ \text{ and } & \ \ ((\varepsilon_{2a}(u), \varepsilon_{2b}(u), \varepsilon_{3a}(u), \varepsilon_{6a}(u)) \in \{(1-2\zeta_6, 1, -1 + 2\zeta_6, 0), (-1 + 2\zeta_6, 1, 1-2\zeta_6, 0) \};  \\
        u^3 \sim 2a \ \ \text{ and } &  \ \ ((\varepsilon_{2a}(u), \varepsilon_{2b}(u), \varepsilon_{3a}(u), \varepsilon_{6a}(u)) \in \{(2\zeta_6, 0, 1 - 2\zeta_6, 0), (2 - 2\zeta_6, 0, -1+2\zeta_6, 0) \}.
        \end{align*}
        For elements $u \in \V(\ZZ[\zeta_{12}]S_5)$ of order $12$ the following remain (always $u^6 \sim 2a$, $u^4 \sim 3a$):
        
        \begin{center}
                \begin{tabular}{ccccccc}\hline
                        $\varepsilon_{2b}(u^3)$ & $\varepsilon_{4a}(u^3)$ & $\varepsilon_{2a}(u^2)$ & $\varepsilon_{3a}(u^2)$ & $\varepsilon_{2b}(u)$ & $\varepsilon_{4a}(u)$ & $\varepsilon_{6a}(u)$ \\ \hline\hline
                        $1-i$ & $i$ & $2\zeta_6$ & $1-2\zeta_6$ & $0$ & $1 + \zeta_{12} + \zeta_{12}^2$ & $ - \zeta_{12} - \zeta_{12}^2$ \\
                        $-i$ & $1+i$ & $2\zeta_6$ & $1-2\zeta_6$ & $1$  & $ \zeta_{12} - \zeta_{12}^2$ & $ - \zeta_{12} + \zeta_{12}^2$ \\
                        $i$ & $1-i$ & $2\zeta_6$ & $1-2\zeta_6$ & $1$  & $- \zeta_{12} - \zeta_{12}^2$ & $\zeta_{12} + \zeta_{12}^2$ \\
                        $1+i$ & $-i$ & $2\zeta_6$ & $1-2\zeta_6$ & $0$  & $1 + \zeta_{12} - \zeta_{12}^2$ & $\zeta_{12} - \zeta_{12}^2$ \\
                        $1-i$ & $i$ & $2-2\zeta_6$ & $-1+2\zeta_6$ & $0$  & $2 -\zeta_{12} - \zeta_{12}^2 + \zeta_{12}^3$ & $-1 + \zeta_{12} + \zeta_{12}^2 - \zeta_{12}^3$ \\ 
                        $-i$ & $1+i$ & $2-2\zeta_6$ & $-1+2\zeta_6$ & $1$ & $-1 -\zeta_{12} + \zeta_{12}^2 + \zeta_{12}^3$ & $1 + \zeta_{12} - \zeta_{12}^2 - \zeta_{12}^3$ \\  
                        $i$ & $1-i$ & $2-2\zeta_6$ & $-1+2\zeta_6$ & $1$ & $-1 +\zeta_{12} + \zeta_{12}^2 - \zeta_{12}^3$ & $1 - \zeta_{12} - \zeta_{12}^2 + \zeta_{12}^3$ \\ 
                        $1+i$ & $-i$ & $2-2\zeta_6$ & $-1+2\zeta_6$ & $0$ & $2 +\zeta_{12} - \zeta_{12}^2 - \zeta_{12}^3$ & $-1 + \zeta_{12} + \zeta_{12}^2 + \zeta_{12}^3$ \\ \hline
                \end{tabular}
        \end{center}
        
Thus (ZC1) holds for $S_5 \times A$, $A$ a finite abelian group if neither
$4$ nor $3$ divides the exponent of $A$. Moreover units of order 4 are
conjugate in $\mathbb{Q}(\zeta)S_5$ if $i \notin \ZZ [\zeta]$, i.e.\ if 4
does not divide the exponent of $A$.    
\end{remark}

\begin{remark}\label{remark_2.S5}
        For $G = 2.S_5 = \SL(2, 5).2$ the HELP method leaves problems with elements of order $8$. Here the problematic distributions of partial augmentations for $u \in \V(\ZZ[i]G)$ of order $8$ are as follows. Let $2a, 4a, 4b, 8a, 8b$ denote the conjugacy classes of order $2$, $4$ and $8$ of $G$ respectively as in the character table \texttt{CharacterTable("2.Sym(5)")} in \textsf{GAP}. Then $u^4 \sim 2a$, $u^2 \sim 4b$ and 
        \begin{align*}(\varepsilon_{4a}(u), \varepsilon_{4b}(u), \varepsilon_{8a}(u), \varepsilon_{8b}(u)) \in \{ & (1-i, 0, 0, i), (i, 1, 0, -i), (-i, 1, i, 0),  (1-i, 0, i, 0), \\
        & (1+i, 0, 0, -i), (i, 1, -i, 0), (-i, 1, 0, i), (1+i, 0, -i, 0) \}
        \end{align*}
        All partial augmentations not stated are zero.
        
         Thus (ZC1) holds for $2.S_5 \times A$, $A$ an abelian group, if $4$ does not divide the exponent of $A$.
\end{remark}

\begin{proposition} \label{specialp} Let $G$ and $H$ be finite groups, $p$ a
  prime and $D = H \times G$. Let $u \in \V(\ZZ D)$ a torsion unit. Let $M$ be a normal $p$-subgroup of $D$ and denote
by $\bar{u}$ the image of $u$ under $\ZZ D \to \ZZ D/M$.
Assume that $o(\bar{u}) < o(u)$ (e.g.\ if $p^m$ does not divide the exponent of $D/M$)
and
that 
$\varepsilon_{D/M[w]}(\bar{u}) \neq 0 $ if, and only if, $w =
o(\bar{u})$. 
Then 
\[ \varepsilon_{D[j]} (u) = 0 , \ \mbox{if} \  j \neq o(u). \]
\end{proposition}

\begin{proof}  
Write $o(u) = p^m \cdot k$ with $k$ coprime to $p$.
As $o(\bar{u}) < o(u)$, we get 
$\varepsilon_g(u) = 0$ for each $g \in D$ whose $p$-part has
smaller order than the $p$-part of $u$, by \cite[Proposition 2]{Hertweck_OTUd}.
Moreover by \ref{knownfacts}.\eqref{HertweckOrder}, $ \varepsilon_g(u) = 0$ provided
$o(g)$ does not divide $o(u)$.    
So $ \varepsilon_g(u) \neq 0 $ implies that $o(g) = p^m \cdot l $ and
$l$ divides $k$.  Looking at the map from $\ZZ D$ onto $\ZZ D/M $ and using \ref{knownfacts}.\eqref{u_p_element}, it follows that
$$ \varepsilon_{D[p^m \cdot l]} (u) = \sum_i \varepsilon_{D/M[p^i \cdot l]}
(\bar{u}).$$ 
 By assumption $\varepsilon_{D/M[p^i \cdot l]}(\bar{u}) = 0$ if 
$p^i \cdot l \neq o(\bar{u})$.
Thus $ \varepsilon_{D[p^m \cdot l]}(u) = 0$ if $l \neq k $ and the
result follows. 
\end{proof}

\begin{corollary} \label{specialorders} Let $A$ be a finite abelian group.  
\begin{itemize}
\item[a)] Let $D = A \times S_4$ and $ u \in \V(\ZZ D)$ with $4$ divides the
  order of $u$. Then
   $$ \varepsilon_{D[m]}(u) = 0 \ \mbox{if} \  m \neq o(u). $$
\item[b)] Let $D = A \times G $ where $G$ is a group of order $48$ mapping
  onto $S_4.$ Let $ u \in \V(\ZZ D)$ with $4$ divides the
  order of $u$. Then
   $$ \varepsilon_{D[m]}(u) = 0 \ \mbox{if} \  m \neq o(u). $$
\item[c)]  Let $D = A \times 2.S_5$   and $ u \in \V(\ZZ D)$ with $8$ divides the
  order of $u .$ Then
   $$ \varepsilon_{D[m]}(u) = 0 \ \mbox{if} \  m \neq o(u). $$
\end{itemize}
\end{corollary}

\begin{proof} In case a) choose $M = \operatorname{O}_2(D) = A_2 \times V_4$, where
  $A_2 $ is the Sylow $2$-subgroup of $A$ and $V_4 = \operatorname{O}_2(S_4)$ denotes
  the Klein $4$-subgroup. Then
  $D/M \cong A_{2'} \times S_3 .$ Clearly $D/M$ has a Sylow tower. Thus 
(Gen-BP) holds for $D/M$ by Proposition \ref{normalSylowp}. Moreover 4 does not
divide the exponent of $A_{2'} \times S_3 .$ So   
Proposition \ref{specialp} completes this case. \\
For the proof of b) note that $|\operatorname{O}_2(G)|\geq 8$ and that $D/\operatorname{O}_2(D)$ is isomorphic to $ A_{2'} \times S_3.$ Thus we
  may argue as in case a). \\ 
     For c) choose $M = A_2 \times Z $, where $Z \cong C_2 $ denotes the
center of $2.S_5 .$  Then $D/M$ is a direct product of an abelian group
of odd order  and $S_5 .$ Clearly $D/M$ has no
elements of order $8$. By Remark~\ref{s5}, units of order dividing 4 of $\V(\ZZ
  D/M)$ are rationally conjugate to elements of $D/M .$ Thus for such
  units the
  generalized trace  $\varepsilon_{D[m]}(u) = 0 \ \mbox{if} \  m \neq o(u)$. 
  Now again Proposition~\ref{specialp} completes this
case.
\end{proof}

We will meet groups with all non-trivial elements of prime order in Proposition~\ref{firstpowerp}. 

\begin{remark}\label{Gpelements}
        If $G$ is a finite group such that every non-trivial element
        of $G$ has prime order, then (ZC1) holds for $G$. In
        particular (Gen-BP) is valid for $G$.  
\end{remark}
\begin{proof}
        By \cite{Chang93}, $G$ is either a $p$-group of exponent $p$,
        a Frobenius group of order $p^a\cdot q$ with $p$ and $q$
        different primes, or it is isomorphic to $A_5$. In all these
        three cases it is well known that (ZC1) holds for $G .$
    This follows e.g.\ from \cite{Weiss91}, \cite[Theorem 5.6]{Hertw06}, \cite{LP89}. 
\end{proof}

\begin{proposition} \label{firstpowerp} Let $D=N \times G$,
  where $N$ is a finite nilpotent group and $G$ is a finite group. 
  Assume that each non-trivial element of $G$ is of prime order. 
Then (Gen-BP) is valid for $D .$
\end{proposition}

\begin{proof} Using Remark~\ref{Gpelements} we get that (Gen-BP) holds for $G$. 
        Let $u$ be a unit of $\V(\ZZ D)$ of prime order $r$. Then by
  \ref{knownfacts}.\eqref{HertweckOrder}, $\varepsilon_{D[m]}(u) = 0$ if $m \neq r$.
  
  Assume that $N$ is a $p$-group. Suppose that $u$ has order $p^k \cdot q$ with $p$ and $q$ different primes dividing $|G|$. 
  By assumption $G$ has no elements of mixed order. Thus it follows
  from \ref{knownfacts}.\eqref{u_p_element} that $u$ maps under $\ZZ D
  \rightarrow \ZZ G$ onto an element of order $q .$ We may 
  apply Proposition~\ref{specialp} with $M = N$ and obtain that 
        \begin{equation}\label{caseP-group}
\varepsilon_{D[m]}(u) = 0 \ \ \mbox{if} \ \ m \neq p^k \cdot q.
        \end{equation} 
  We claim that every element of $\V(\ZZ D)$ has order $p^k \cdot q$ with $p$ and $q$ different primes dividing $|G|$. Indeed, assume otherwise that $v\in \V(\ZZ D)$ has order $p^k \cdot q \cdot r$, with $r$ a prime different from $p$ and $q$. Then projecting $v$ via the natural homomorphism $\ZZ D \rightarrow \ZZ G$ we get an element of $\V(\ZZ G)$ whose order is multiple of $p\cdot q$, in contradiction with the assumptions on $G$. This finishes the proof of the claim. Therefore, combining \eqref{caseP-group} with the claim we deduce that (Gen-BP) follows in this
case.

We proceed by induction on the number of primes dividing $|N|. $ Let
$P$ be the Sylow $p$-subgroup and let $N = P \times M$.
   If $p\nmid o(u)$ then by \ref{knownfacts}.\eqref{HertweckOrder} we have 
$\varepsilon_{D[m]}(u) = \varepsilon_{D/P[m]}(\bar{u})$ for each $m$
dividing $|D/P|$, where
$\bar{u}$ is the image of $u$ under the map $\ZZ D \rightarrow \ZZ D/P.$
So by induction $\varepsilon_{D[m]}(u) = 0 $ if, and only if, $m =
o(u) .$ 
Suppose now that $o(u) = p^l \cdot k $ with $l \geq 1.$
If $l \geq 2$ then by Proposition~\ref{specialp} we get 
$$\varepsilon_{D[m]}(u) = 0, \ \mbox{if} \ m \neq o(u) .$$
Note that
the arguments work for each prime dividing $|N| .$ 
Thus it suffices to consider units of order $p \cdot q \cdot k $,
where
$p$ and $q$ are different primes and $k$ is square-free and coprime to $p \cdot q$.
Let $P$ and $Q$ be the Sylow subgroups of $N$ corresponding to $p$ and
$q$.

Denote the image of $u$ in $\ZZ D/P$ by $u_1$, in $\ZZ D/Q$ by $u_2$ and that one in 
$\ZZ D/(P \cdot Q)$ by $v$. Suppose that $u_1$ and $u_2$ both have order divisible by $p \cdot q$.
As $v$ is the image of $u_1$ under $\ZZ D/P \rightarrow \ZZ D/(P \cdot Q)$ and torsion units 
mapping to $1$ under this homomorphism are $q$-elements by \ref{knownfacts}.\eqref{u_p_element},
$v$ has order divisible by $p$. Similarly, looking at $u_2$, one gets that $v$ 
has order divisible by $q$. But the we also get a torsion unit of order $p\cdot q$ 
in $\V(\ZZ G)$, for which (Gen-BP) holds by assumption. This contradiction shows that
either $u_1$ or $u_2$ have order not divisible by $p \cdot q$.

W.l.o.g.\ we assume that $p \cdot q $ does not divide $o(u_1)$. 
Thus we may apply Proposition~\ref{specialp} with $M = P$.  Note  that by induction (Gen-BP) holds for $D/P .$    
Consequently 
$\varepsilon_{D[m]}(u) = 0 \ \ \mbox{ if} \ \ m \neq o(u)$.
\end{proof}

\begin{theorem} Let $G$ be a Frobenius group and let $A$ be a finite 
  abelian group. Then (Gen-BP) holds for $G \times A$. 
\end{theorem}  

\begin{proof}
Using Remark~\ref{remark_on_Frobenius} and Remark~\ref{remarkSL25_S5_2S5} 
we get the following.  Let $C$ be a Frobenius complement of
  $G$. Then (Gen-BP) is valid for $G \times
  A$ provided this is the case for $C \times B$, where $B$ is abelian.
  If $C$ is metabelian then $C \times B$ is metabelian and (Gen-BP)
  holds by \cite[Corollary 1.4]{DokSeh}. If $C$ is not metabelian then as
  explained in the proof of Proposition \ref{Fcomplement} it has a
  Sylow tower if it does not map onto $\SL (2,3).2 $ or $\SL( 2,5).2$. \\
  If $C$ maps onto  $\SL (2,3).2$ or 
  onto $\SL(2,5).2$ then by Proposition \ref{ExtendedResult} and
  Remark~\ref{remark_2.S5} for $\SL(2,5).2$ 
 we see that 
the generalized trace of torsion units
  behaves as desired when the order of the unit is not
  divisible in the first case by $4$ or in the second case by
  $8$ (even rational conjugacy holds in these cases). 
Now Corollary \ref{specialorders} b) and c) establish the theorem.   
\end{proof}

\begin{theorem} Let $G$ be a  group with $|G|\leq 95$  and let
  $A$ be a finite
  abelian group. Then (Gen-BP) holds for $G \times A$.   
\end{theorem}

\begin{proof} By Proposition \ref{ExtendedResult} even the first
  Zassenhaus conjecture holds for $G \times A$ provided $G$ does not map onto $S_4$
  or $G$ does not coincide with $S_3 \wr C_2$ or $A_5$. \\
$S_3 \wr C_2 $ is a Sylow tower group thus also $A \times (S_3 \wr C_2) $
  is a Sylow tower group and Proposition~\ref{normalSylowp} yields the result. \\
(Gen-BP) holds for $A_5 ,$ cf. Remark \ref{Gpelements}. 
Now Proposition \ref{firstpowerp} establishes this case. \\
Assume now that $G$ maps onto $S_4. $   
By Proposition \ref{ExtendedResult} we know that each torsion unit whose order is not divisible by
  4 has
the desired generalized traces. Corollary \ref{specialorders} a) and b)
  rsp. 
  prove the theorem for $S_4  $ and when $G$ is a group of order $48$ 
  mapping
  onto $S_4 .$ \\
    Finally let $G$ be a group of order $72$ with $S_4$ as
  image. Then $G$ has a minimal normal subgroup $M$ isomorphic to $C_2
  \times C_2 .$ We have that  (ZC1) holds for $G/M$ and also $G/M$ has no elements of
  order 4. Thus Proposition~\ref{specialp} completes the proof. 
\end{proof}

\noindent\textbf{Acknowledgment}. We are very thankful to {\'A}ngel del R{\'i}o for
several useful conversations.


\begin{thebibliography}{99}
\itemsep=-2pt

\bibitem{Ari}
\emph{Mini-{W}orkshop: {A}rithmetik von {G}ruppenringen}, Oberwolfach Rep.
  \textbf{4} (2007), no.~4, 3209--3239, Abstracts from the mini-workshop held
  November 25--December 1, 2007, Organized by Eric Jespers, Zbigniew Marciniak,
  Gabriele Nebe and Wolfgang Kimmerle, Oberwolfach Reports. Vol. 4,
  no. 4. \url{https://www.mfo.de/occasion/0748c/www_view}

\bibitem{BHKMS}  A.~B\"achle, A.~Herman, A.~Konovalov, L.~Margolis and
  G.~Singh. \emph{The Status of the Zassenhaus Conjecture for Small
  Groups}, Experiment. Math., to appear, 6 pages,
  2018. \url{http://dx.doi.org/10.1080/10586458.2017.1306814}.  

\bibitem{BKM} A.~B\"achle, W.~Kimmerle and L.~Margolis, \emph{Algorithmic
    aspects of units in group rings} to appear in Algorithmic and Experimental Methods in Algebra, 
    Geometry, and Number Theory, G. B{\"o}ckle, W. Decker, G. Malle (eds.), Springer Verlag, 2018.
    \href{https://arxiv.org/abs/1612.06171}{\nolinkurl{arXiv:1612.06171 [math.RT]}}.   
  
\bibitem{HeLPOverview}
A.~B\"achle and L.~Margolis, \emph{{HeLP} -- {A} \textsf{GAP} package for torsion units in integral
  group rings}, preprint,
  \href{http://arxiv.org/abs/1507.08174}{\nolinkurl{arXiv:1507.08174
  [math.RT]}} (2015), 6 pages.  
  
\bibitem{BMM10}
A.~B\"achle and L.~Margolis, \emph{Rational conjugacy of torsion units in
  integral group rings of non-solvable groups},
   Proc. Edinb. Math. Soc. (2) \textbf{60}(3), 813-830, 2017.

\bibitem{Bov} A.A.~Bovdi, \emph{The unit group of an integral group ring}
  (Russian), Uzhgorod Univ. Uzhgorod 1987. 

\bibitem{BovHer}
V.A.~Bovdi and M.~Hertweck.
\newblock \emph{Zassenhaus conjecture for central extensions of {$S_5$}}.
\newblock J. Group Theory, 11(1):63--74, 2008.
  
\bibitem{CMdR} M.~Caicedo, L.~Margolis and \'A.~del~Rio, \emph{Zassenhaus
      conjecture for cyclic-by-abelian groups},  J. Lond. Math. Soc. (2)
      {\bf 88} (2013), no.1, 65-78.
      
\bibitem{Chang93} K.N.~Cheng,  M.~Deaconescu,   M.~Lang, and W.J.~Shi. \emph{Corrigendum and addendum to: {C}lassification of finite  groups with all elements of prime order}, Proc. Amer. Math. Soc. (1993), 117 (1205--1207).       
  
\bibitem{CohLiv} J.A.~Cohn and D.~Livingstone,
     \emph{On the structure of group algebras. I} 
Canad. J. Math. \textbf{17} (1965), 583--593.

\bibitem{DarkScoppola} R.~Dark and C.M.~Scoppola, \emph{On Camina group of prime power order} J. Algebra \textbf{181} (1996), 787-802.
  
\bibitem{DoJu} M.A.~Dokuchaev, S.O.~Juriaans, \emph{Finite subgroups in integral group rings},
Canad. J. Math. \textbf{48}(6):1170--1179, 1996.

\bibitem{DokSeh} M.A.~Dokuchaev, S.K.~Sehgal, \emph{Torsion units in
    integral group rings of solvable groups}, Comm. Algebra \textbf{22}
    (1994), no.12, 5005--5020.

\bibitem{DJPM}
M.A.~Dokuchaev, S.O.~Juriaans, and C.~Polcino~Milies.
\newblock \emph{Integral group rings of {F}robenius groups and the conjectures of
  {H}. {J}. {Z}assenhaus}.
\newblock Comm. Algebra, 25(7):2311--2325, 1997.

\bibitem{EisMar} F.~Eisele, L.~Margolis, \emph{A Counterexample to the First Zassenhaus Conjecture},
 \href{https://arxiv.org/abs/1710.08780}{\nolinkurl{arXiv:1710.08780 [math.RA]}}, 2017.


\bibitem{GAP4}
The GAP~Group, \emph{{GAP -- Groups, Algorithms, and Programming, Version
  4.8.3}}, 2016, \url{http://www.gap-system.org}.

\bibitem{Hertw06} M.~Hertweck, \emph{Torsion units in integral group rings of certain metabelian groups}, Algebra Colloq., {\bf 13} (2), (2006), 329-348. 

\bibitem{HertweckPa} M.~Hertweck, \emph{Partial augmentations and {B}rauer character values of torsion units in group rings},
\href{https://arxiv.org/abs/math/0612429v2}{\nolinkurl{arXiv:math/0612429v2 [math.RA]}}, 2007.

\bibitem{Hertweck08} M.~Hertweck, \emph{Torsion units in integral group rings of certain metabelian groups}, Proc. Edinb. Math. Soc. (2) {\bf 51} (2008), no. 2, 363-385.

\bibitem{Hertweck_OTUd}
M.~Hertweck, \emph{The orders of torsion units in integral group rings of
  finite solvable groups}, Comm. Algebra \textbf{36} (2008), no.~10,
  3585--3588.

\bibitem{Hertweck12} M.~Hertweck, \emph{On torsion units in integral group rings of Frobenius groups},
\href{https://arxiv.org/abs/1207.5256v1}{\nolinkurl{arXiv:1207.5256v1 [math.RT]}}. 2012.

\bibitem{Hertweck13} M.~Hertweck, \emph{A criterion for $p$-adic conjugacy of $p$-torsion units in finite group rings}, manuscript, 2013.

\bibitem{Hoe} C.~H\"ofert, \emph{Die erste Vermutung von Zassenhaus
    f\"ur Gruppen kleiner Ordnung}, Diplomarbeit, University of
  Stuttgart (2004).

\bibitem{HoeKim} C.~H\"ofert, W.~Kimmerle, \emph{On torsion units of interal group rings of groups of small order}, Lect. Notes Pure Appl. Math., \textbf{248}, Chapman \& Hall/CRC, Boca Raton, FL, 2006. 

\bibitem{Huppert67} B.~Huppert, \emph{Endliche Gruppen I.} Springer-Verlag, Berlin, 1967.

\bibitem{isaacslewis} I.M.~Isaacs, M.~Lewis, \emph{Camina {$p$}-groups that are generalized {F}robenius complements}, Arch. Math. (Basel), \textbf{104}(5), 401--405, 2015.

\bibitem{JuriaansPolcino00} S.O.~Juriaans and C.P.~Milies, \emph{Units of integral group rings of Frobenius groups},
J. Group Theory, (3):277--284, 2000.

\bibitem{Lewis14} M.~Lewis, \emph{Classifying Camina groups: a theorem of Dark and Scoppola}, Rocky Mountain J. Math. 44 (2014), 591-597.

\bibitem{LP89} I.S.~Luthar, I.B.S.~Passi, \emph{Zassenhaus conjecture for A5}. Proc. Indian Acad. Sci. Math. Sci.  (1989) 99(1):1-5.

\bibitem{LT91}
I.S.~Luthar, P.~Trama, \emph{Zassenhaus conjecture for $S_5$}. Comm. Algebra \textbf{19} (1991), no.~8, 2353-2362.   

\bibitem{MRSW} Z.~Marciniak, J.~Ritter, S.K.~Sehgal, A.~Weiss, \emph{Torsion units in integral group rings of some metabelian groups. {II}},
Journal of Number Theory, 25(3):340--352, 1987.

\bibitem{MarHer} L.~Margolis, \emph{A theorem of Hertweck on $p$-adic conjugacy of $p$-torsion units in group rings}, 
  \href{https://arxiv.org/abs/1706.02117v1}{\nolinkurl{arXiv:1706.02117v1 [math.RA]}} (2017), 13 pages.

\bibitem{MdR} L.~Margolis, \'A.~del R\'io, \emph{Partial augmentations
    property: A Zassenhaus conjecture related problem}, 
  \href{http://arxiv.org/abs/1706.04787v2}{\nolinkurl{arXiv:1706.04787v2
  [math.RA]}} (2017), 13 pages.  

\bibitem{Passman68} D.~Passman, \emph{Permutation groups}. W. A. Benjamin, Inc., New York-Amsterdam 1968.

\bibitem{AngelMariano} {\'A}.~del R{\'{\i}}o, and M.~Serrano, \emph{On the torsion units of the integral group ring of finite projective special linear groups}. Comm. Algebra \textbf{45} (2017), no.~2, 5073-5087.


\bibitem{SehgalBook} S.K.~Sehgal, \emph{Units in Integral Group Rings}, volume 69 of Pitman Monographs and Surveys in Pure and Applied
Mathematics. (1993) Harlow: Longman Scientic $\&$ Technical.

\bibitem{Weiss91} A.~Weiss, \emph{Torsion units in integral group rings}, J. Reine Angew. Math. 415 (1991), 175-187.

\bibitem{WeissAN} E.~Weiss, \emph{Algebraic number theory}, McGraw-Hill Book Co., Inc., New York-San Francisco-Toronto-London, 1963.

\bibitem{Zassenhaus} H.J.~Zassenhaus, \emph{On the torsion units of finite group rings}, Studies in mathematics 
(in honor of {A}. {A}lmeida {C}osta) ({P}ortuguese). Instituto de Alta Cultura, Lisbon. (1974) 119--126. 

\end{thebibliography}
\end{document}